\documentclass[smallextended,numbook,runningheads]{svjour3}     
\smartqed  
\usepackage{graphicx,amsmath,amssymb,bm,color,geometry,subfig}
\usepackage{eucal}
\usepackage{lineno}
\usepackage{algorithmic,algorithm,comment}
\usepackage{tabularx}
\usepackage{booktabs}

\journalname{}

\newcommand{\norv}{N}               
\newcommand{\svc}{\mu}        		
\newcommand{\pvc}{x}          		

\newcommand{\sv}{\bm{\svc}}    		
\newcommand{\pv}{\bm{\pvc}}    		
\newcommand{\pdf}{\rho}    	        
\newcommand{\pdfc}{\rho}            
\newcommand{\nstoch}{n_{\svc}}      
\newcommand{\nphy}{n_{\pvc}}        

\newcommand{\gpc}{\phi}             
\newcommand{\GPC}{\Phi}             

\newcommand{\ap}{{\scriptsize\mbox{ap}}}  %

\newcommand{\rd}{\mathrm{d}}        
\newcommand{\normp}[1]{\Vert#1\Vert_{L^2}}  
\newcommand{\tol}{\mathtt{TOL}}
\newcommand{\rTheta}{\mathrm{\Theta}}
\newcommand{\err}{\mathtt{ERR}}
\newcommand{\refs}{\mathtt{REF}}

\newcommand{\mean}[1]{\mathrm{E}\left[#1\right]}
\newcommand{\var}[1]{\mathrm{V}\left[#1\right]}
\newcommand{\anovat}{\mathbb{T}}
\newcommand{\anovaT}{\mathfrak{T}}
\newcommand{\anovas}{\mathbb{S}}
\newcommand{\anovau}{\mathbb{U}}
\newcommand{\anovam}{\mathbb{M}}
\newcommand{\anovap}{\mathbb{P}}

\def\dsR{\hbox{R}}
\def\dsN{\hbox{N}}

\newcommand{\guanjie}[1]{{\color{black}#1}}

\begin{document}

\title{An adaptive ANOVA stochastic Galerkin method for partial differential equations with \guanjie{high-dimensional} random inputs
}

\titlerunning{An adaptive ANOVA stochastic Galerkin method for PDEs with high-dimensional random inputs}        

\author{Guanjie Wang \and Smita Sahu \and Qifeng Liao
}


\institute{
Guanjie Wang  \at School of Statistics and Mathematics,
Shanghai Lixin University of Accounting and Finance, Shanghai, China.\\\ \email{guanjie@lixin.edu.cn}  \\\  
Smita Sahu \at School of Mathematics and Physics, University of Portsmouth, Lion Terrace, PO1~3HF, UK.\\\
\email{smita.sahu@port.ac.uk}\\\ 
Qifeng Liao \at School of Information Science and Technology,
ShanghaiTech University, Shanghai, China.\\\ \email{liaoqf@shanghaitech.edu.cn 
} 
}

\date{Received: date / Accepted: date}
\maketitle

\begin{abstract}
It is known that standard stochastic Galerkin methods  encounter challenges when solving partial differential equations with high-dimensional random inputs, which are typically caused by the large number of stochastic basis functions required. It becomes crucial to properly choose effective basis functions, such that the dimension of the stochastic approximation space can be reduced.  In this work, we focus on the stochastic Galerkin approximation associated with generalized polynomial chaos (gPC), and explore the gPC expansion based on the analysis of variance (ANOVA) decomposition.  A concise form of the gPC expansion is presented for each component function of the ANOVA expansion, and an adaptive ANOVA procedure is proposed to construct the overall stochastic Galerkin system. Numerical results demonstrate the efficiency of our proposed adaptive ANOVA stochastic Galerkin method \guanjie{for both  diffusion and Helmholtz problems}.

\end{abstract}

\section{Introduction}

Over the past few decades, there has been a significant increase in efforts to develop efficient uncertainty quantification approaches for solving partial differential equations (PDEs) with random inputs. Typically, these random inputs arise from a lack of precise measurements or a limited understanding of realistic model parameters, such as permeability coefficients in diffusion problems and refraction coefficients in acoustic problems~\cite{Xiu2002modeling,Elman2005,Feng2015}.

Designing a surrogate model or calculating statistics (such as mean and variance of the solution) for partial differential equations (PDEs) with random inputs is of great interest, especially when the inputs are high-dimensional. To achieve this, extensive efforts have been made. The Monte Carlo method (MCM) and its variants are among the direct methods for computing the mean and variance~\cite{Caflisch1998,fishman2013}. In MCM, numerous sample points of the random inputs are generated based on their probability density functions. 
For each sample point, the corresponding deterministic problem can be solved using existing numerical methods.
The statistics of the stochastic solution can then be estimated by aggregating the results of these deterministic solutions. While MCM is easy to implement, it converges slowly and typically requires a large number of sample points. Additionally, it does not provide a surrogate model directly, which limits its applications. 

To enhance efficiency, the stochastic collocation method (SCM) and the stochastic Galerkin method (SGM) have been developed. Both the SCM and the SGM are typically more efficient than the MCM for solving partial differential equations (PDEs) with moderate dimensional random inputs~\cite{Xiu05,Xiu2002modeling,Xiu2002wiener,Xiu2003modeling,tang2010convergence,yan2019adaptive}. 
To further accelerate the SCM and the SGM, various of methods such as the reduced basis collocation method~\cite{Elmanliao},  the dynamically orthogonal approximation~\cite{Cheng2013a,Cheng2013b,musharbash2015error}, the reduced basis solver based on low-rank approximation~\cite{powell2015} and the preconditioned low-rank projection methods~\cite{LeeElman16,Lee2019low} are actively studied. However, these methods still face challenges in addressing high-dimensional problems, as the number of collocation points required by the SCM and the number of unknowns in the SGM increase rapidly with an increasing number of random variables, a well-established phenomenon referred to as the curse of dimensionality.

To address the challenges posed by high-dimensional problems, novel techniques have been developed and implemented. For example, the adaptive sparse grids~\cite{Agarwal2009domain}, multi-element generalized polynomial chaos~\cite{Wan2005adaptive}, the compressive sensing approaches~\cite{guo2020constructing,jakeman2017generalized}, and anchored ANOVA methods (i.e. cut-HDMR)~\cite{Ma2010adaptive,Yang2012adaptive}. In particular, the anchored ANOVA method has been extensively employed in various research studies (see for instance~\cite{Tang2015sensitivity,Gao2010anova,Sudret2008global,Tang2016,Liao2016reduced}). It is shown that the choice of the anchor point is crucial for efficient approximation~\cite{Sobol2003theorems,Wang2008approximation}. The work \cite{Tang2015sensitivity} proposes to use the covariance decomposition to effectively evaluate the output variance of multivariate functions. An efficient approximation strategy for high-dimensional periodic functions is proposed based on the fast Fourier transform and the ANOVA decomposition in \cite{Daniel2021}, and its application to PDEs with random coefficients is studied in~\cite{Fabian2022}. \guanjie{The studies conducted in~\cite{cho2018adaptive,williamson2021application} explore the adaptive reduced basis collocation method based on ANOVA decomposition and its applications to problems involving anisotropic random inputs and  stochastic Stokes-Brinkman equations.}

In this paper, we investigate the generalized polynomial chaos (gPC) expansion of component functions for the ANOVA decomposition,  and  present a concise form of the gPC expansion for each component function. With this formulation, we propose an adaptive ANOVA stochastic Galerkin method. The proposed method adaptively selects the effective gPC basis functions in the stochastic space, reducing the dimension of the stochastic approximation space significantly, and leveraging the orthonormality of the gPC basis to facilitate the computation of the variance of each term in the ANOVA decomposition.  \guanjie{Note that compared with anchored ANOVA collocation methods \cite{Ma2010adaptive,Yang2012adaptive,Liao2016reduced,cho2018adaptive}, our proposed adaptive ANOVA stochastic Galerkin method avoids the difficulty for selecting anchor points, which are crucial for anchored ANOVA methods \cite{Sobol2003theorems,Wang2008approximation,Gao2010anova}.}  Additionally, the proposed method provides a straightforward approach to build a surrogate model. We conduct numerical simulations and present the results to demonstrate the effectiveness and efficiency of our proposed method.

An outline of the paper is as follows. We present our problem setting in the next section. In Section~\ref{sec:SGANOVA}, we review the stochastic Galerkin method and the ANOVA decomposition for partial differential equations with random inputs. Our main theoretical results and the adaptive ANOVA stochastic Galerkin method are presented in Section~\ref{sec:Adaptive SGANOVA}. Numerical results are discussed in Section~\ref{sec:numericaltests}. Section~\ref{sec:concusion} concludes the paper. 

\section{Problem setting}
Let $D\subseteq {\dsR}^d~(d = 2, 3)$ denote a physical domain that is bounded, connected, with a polygonal boundary $\partial D$, and $\pv \in {\dsR}^d$ denote a physical variable.
Let $\sv = (\svc_1,\ldots,\svc_{\norv})$ be a random vector of dimension of $\norv$, where the image of $\svc_i$ is denoted by $\Gamma_i$, and the probability density function of $\svc_i$  is denoted by $\rho_{i}(\svc_i)$. We further assume that the components of $\sv$, i.e., $\svc_1,\ldots,\svc_{\norv}$ are mutually independent, then the image of $\sv$ is given by $\Gamma = \Gamma_1\times\cdots\times\Gamma_{\norv}$, and the probability density function of $\sv$ is given by $\pdf(\sv) = \prod_{i=1}^{\norv}\pdfc_{i}(\svc_i)$. In this work, we focus on the partial differential equations (PDEs) with random inputs, that is
\begin{equation}\label{eq:spde}
\begin{cases}
\mathfrak{L}(\pv,\sv,u(\pv,\sv)) = f(\pv)\ & \forall~ (\pv,\sv)\in D\times\Gamma,\\
\mathfrak{b}(\pv,\sv,u(\pv,\sv)) = q(\pv)\ &\forall~ (\pv,\sv)\in\partial D\times \Gamma,
\end{cases}
\end{equation}
where $\mathfrak{L}$ is a linear partial differential operator with respect to physical variables, and $\mathfrak{b}$ is a boundary operator. Both  operators can have random coefficients. The source function is denoted by $f(\pv)$, and $q(\pv)$ specifies the boundary conditions. Additionally, we assume that $\mathfrak{L}$ and $\mathfrak{b}$ are affinely dependent on the random inputs. Specifically, we have
\begin{eqnarray}
\mathfrak{L}(\pv,\sv,u(\pv,\sv)) = \sum_{i=1}^{K}\rTheta_{\mathfrak{L}}^{(i)}(\sv)\mathfrak{L}_i(\pv,u(\pv,\sv)), \label{eq:affine1}\\ 
\mathfrak{b}(\pv,\sv,u(\pv,\sv)) = \sum_{i=1}^{K}\rTheta_{\mathfrak{b}}^{(i)}(\sv)\mathfrak{b}_i(\pv,u(\pv,\sv)),\label{eq:affine2}
\end{eqnarray} 
where $\{\mathfrak{L}_{i}\}_{i=1}^{K}$ are parameter-independent  linear differential operators, and $\{\mathfrak{b}_{i}\}_{i=1}^{K}$ are parameter-independent  boundary operators. Both $\rTheta_{\mathfrak{L}}^{(i)}(\sv)$ and $\rTheta_{\mathfrak{b}}^{(i)}(\sv)$ take values in $\dsR$ for $i=1,\ldots,K$.

It is of interest to design a surrogate model for the problem \eqref{eq:spde} or calculate statistics of the stochastic solution $u(\pv,\sv)$, such as the mean and the variance. 
\section{Stochastic Galerkin method and ANOVA decomposition}\label{sec:SGANOVA}
In this section, we introduce the stochastic Galerkin methods for solving problem~\eqref{eq:spde}, and  we review the ANOVA decomposition for multi-variable functions. For the sake of presentation simplicity, we consider problems that satisfy homogeneous Dirichlet boundary conditions. However, it is noteworthy that the approach we present can be readily extended to \guanjie{other arbitrary (well-posed) boundary conditions.}
\subsection{Variational formulation}
To introduce the variational form of \eqref{eq:spde}, some notations are required. We first define the Hilbert spaces $L^2(D)$ and $L^2_{\pdf}(\Gamma)$ via
\begin{eqnarray*}
	L^2(D)&:=&\left\{v(\pv): D \to \dsR\ \bigg|\ \int_D v^2(\pv)\, \rd\pv <\infty \right\}, \\
	L_{\pdf}^2(\Gamma)&:=&\left\{g(\sv): \Gamma \to \dsR\ \bigg|\ \int_{\Gamma}\pdf(\sv) g^2(\sv)\, \rd\sv <\infty \right\},
\end{eqnarray*} 
which are equipped with the inner products
\begin{eqnarray*}\label{eq:inner1}
\langle v(\pv),\hat{v}(\pv)\rangle_{L^2}&:=&\int_D v(\pv)\hat{v}(\pv)\,\rd\pv, \\
\langle g(\sv),\hat{g}(\sv)\rangle_{L^2_\pdf}&:=&\int_{\Gamma}\pdf(\sv) g(\sv)\hat{g}(\sv)\,\rd\sv.\label{eq:inner2}
\end{eqnarray*}
Following presentation from Babu\v ska et al.\cite{Babuska2004}, we define the tensor space of $L^2(D)$ and $L^2_{\pdf}(\Gamma)$ as
\begin{equation}\notag
L^2(D)\otimes L^2_{\pdf}(\Gamma):=\left\{ w(\pv,\sv) \bigg| w(\pv,\sv)=\sum_{i=1}^{n}v_i(\pv)g_i(\sv), v_i(\pv)\in L^2(D),\, g_i(\pv)\in L^2_{\pdf}(\Gamma), n\in \dsN \right\},
\end{equation}
which is equipped with the inner product
\begin{equation}\notag
\left\langle w(\pv,\sv),\hat{w}(\pv,\sv)\right\rangle_{L^2\otimes L^2_{\pdf}} =\int_\Gamma\int_{D} w(\pv,\sv)\hat{w}(\pv,\sv)\pdf(\sv) \rd\pv \rd\sv.
\end{equation}
We next define the space
\begin{equation}\notag
H_0^1(D):=\left\{v\in H^1(D)\,| \,  v=0 \textrm{ on } \partial D\right\},
\end{equation}
where $H^1(D)$ is the  Sobolev  space 
\begin{equation}\notag
H^1(D):=\left\{v\in L^2(D)\,, \,   \partial v/ \partial x_i\in L^2(D),  i=1,\ldots,d\right\}.
\end{equation}
\guanjie{Furthermore,  we define the solution and test function space
	\begin{equation}\notag
		W:= H_0^1(D) \otimes L_{\pdf}^2(\Gamma)=\left\{w(\pv ,\sv )\in H_0^1(D)\otimes L^2_{\pdf}(\Gamma) \bigg|
		\Vert w(\pv ,\sv )\Vert_{L^2\otimes L^2_{\pdf}} <\infty \right\}, 
	\end{equation}
	where $\Vert\cdot\Vert_{L^2\otimes L^2_{\pdf}}$ is the norm induced by the inner product $\langle\ \cdot\ ,\ \cdot\ \rangle_{L^2\otimes L^2_{\pdf}}$.}
The variational form of \eqref{eq:spde} can be written as: find $u$ in $W = H_0^1(D)\otimes L^2_{\pdf}(\Gamma )$ such that 
\begin{equation}\label{eq:weak}
\mathfrak{B}(u,w) = \mathfrak{F}(w),\ \forall~ w\in W,
\end{equation}
where 
\begin{equation}\notag
\mathfrak{B}(u,w):=\left\langle \mathfrak{L}(\pv,\sv,u(\pv,\sv)),w(\pv,\sv)\right\rangle_{L^2\otimes L^2_\pdf}, 
\ \mathfrak{F}(w):= \left\langle f(\pv),w(\pv, \sv) \right\rangle_{L^2\otimes L^2_\pdf}.
\end{equation}
Since $\mathfrak{L}$ is affinely  dependent on the parameter $\sv\in\Gamma$ (see \eqref{eq:affine1}) then $\mathfrak{B}$ has the following form
\begin{equation}\label{eq:affine3}
\mathfrak{B}(u,w)= \sum_{i=1}^{K}\mathfrak{B}_i(u,w),
\end{equation}
where the component bilinear forms $\mathfrak{B}_i(\cdot, \cdot)$ for $i\in \dsN^+$ are defined as
\begin{equation}\label{eq:subweak}
\mathfrak{B}_i(u,w):=\left\langle\rTheta_{\mathfrak{L}}^{(i)}(\sv)\mathfrak{L}_i(\pv,u(\pv,\sv)), w(\pv,\sv) \right\rangle_{L^2\otimes L^2_\pdf}.
\end{equation}
\subsection{Discretization}\label{subsec:weakform}
A discrete version of (\ref{eq:weak}) is obtained by introducing a finite dimensional subspace to approximate~$W$. 
Specifically, we first denote the finite dimensional subspaces 
of the corresponding stochastic and physical spaces by 
\begin{equation}\notag
\begin{split}
S_p= {\rm span}\left\{\Phi_j(\sv)\right\}_{j=1}^{\nstoch}\subseteq L^2_{\pdf}(\Gamma)\},\
V_h = {\rm span}\left\{v_s(\pv)\right\}_{s=1}^{\nphy}\subseteq H_0^1(D)\},
\end{split}
\end{equation}
where $\Phi_j(\sv)$ and $v_s(\pv)$ refer to basis functions. We next define a finite dimensional subspace of the overall solution (and test) function space $W$ by
\begin{equation}\notag
W_h^p:= V_h \otimes S_p :={\rm span} \left\{v(\pv)\Phi(\sv)\left| v\in V_h, \Phi\in S_p \right.\right\}.
\end{equation}
The stochastic Galerkin method seeks an approximation $u^{\ap}(\pv,\sv)\in W_h^p$ such that 
\begin{equation}\label{eq:weak2}
\mathfrak{B}(u^{\ap},w) = \mathfrak{F}(w),\ \forall~ w\in W_h^p.\notag
\end{equation}
Suppose $u^{\ap}(\pv,\sv)$ is defined as
\begin{equation}\label{eq:appsg}
u^{\ap}(\pv,\sv):= \sum_{s=1}^{\nphy}\sum_{j=1}^{\nstoch}u_{sj}\Phi_j(\sv)v_s(\pv).
\end{equation}
Since $\mathfrak{L}$ is affinely dependent on the random inputs (see \eqref{eq:affine1}), we substitute \eqref{eq:appsg} into \eqref{eq:subweak} to obtain
\begin{equation}\label{eq:linerop}
\mathfrak{B}_i(u^{\ap}(\pv,\sv),w) = \sum_{s=1}^{\nphy}\sum_{j=1}^{\nstoch}u_{sj}\Phi(\sv)\left\langle\rTheta_{\mathfrak{L}}^{(i)}(\sv)\mathfrak{L}_iv_s(\pv), w(\pv,\sv)\right\rangle_{L^2\otimes L^2_\pdf}.
\end{equation}
Combining \eqref{eq:linerop}  with \eqref{eq:weak}--\eqref{eq:affine3}, we obtain a linear system for the unknown coefficients $u_{sj}$: 
\begin{equation}\label{eq:linsg}
\left(\sum_{i=1}^{K}\bm{G}_i\otimes \bm{A}_i\right)\bar{\bm{u}} = \bm{h}\otimes\bm{f}, 
\end{equation}
where $\{\bm{G}_i\}_{i=1}^K$ are matrices of size $\nstoch\times \nstoch$, and $\bm{h}$ is a column vector of length $\nstoch$. They are defined via  
\begin{equation}\label{eq:stoch_mat}
\bm{G}_i(j,k) = \langle \rTheta_{\mathfrak{L}}^{(i)}(\sv)\Phi_j(\sv),\Phi_k(\sv) \rangle_{L^2_\pdf}, \ \bm{h}(i) = \langle\Phi_i(\sv),1 \rangle_{L^2_\pdf}.
\end{equation}
The matrices $\bm{A}_i$ and the vector $\bm{f}$ in \eqref{eq:linsg} are defined through
\begin{equation}\label{eq:phy_mat}
\bm{A}_i(s,t) = \left\langle\mathfrak{L}_iv_s,v_t \right\rangle_{L^2},\ \bm{f}(s) = \left\langle f,v_s \right\rangle_{L^2}, 
\end{equation}
where $s = 1,\ldots\nphy$ and $t = 1,\ldots, \nphy$. The vector $\bar{\bm{u}}$ in \eqref{eq:linsg} is a column vector of length~$\nphy\times\nstoch$, and is defined by
\begin{equation}\notag
\bar{\bm{u}}=\begin{bmatrix}
\bm{u}_1 \\
\vdots\\
\bm{u}_{\nstoch}
\end{bmatrix}, \ \mbox{where}\ 
\bm{u}_j=\begin{bmatrix}
{u}_{1j} \\
\vdots\\
{u}_{\nphy j}
\end{bmatrix}, \ j= 1,\ldots, \nstoch.
\end{equation}

\subsection{ANOVA decomposition}
We define some notations before introducing the ANOVA decomposition. Let $\anovat = \{t_1,\ldots,t_{|\anovat|}\}$ be a subset of $\anovau=\{1,\ldots,\norv\}$, where $|\anovat|$ is the cardinality of $\anovat$. For special case where $\anovat = \emptyset$, we set~$|\anovat|$ to $0$. Otherwise, we assume that $t_1< t_2<\ldots<t_{|\anovat|}$. In addition, for $\anovat\neq \emptyset$, let $\sv_{\anovat}$ denote the $|\anovat|$-vector that contains the components of the vector $\sv$ indexed by~$\anovat$, i.e., $\sv_{\anovat}=(\svc_{t_1},\ldots,\svc_{t_{|\anovat|}})$. \guanjie{Furthermore, let $\pdf_{\anovat}(\sv_{\anovat})$ and $\Gamma_{\anovat}$ denote the probability density function and the image corresponding to $\sv_{\anovat}$ respectively, i.e.,
\begin{equation}\notag
\pdf_{\anovat}(\sv_{\anovat}) = \pdf_{t_1}(\svc_{t_1})\cdots\pdf_{t_{|\anovat|}}(\svc_{t_{|\anovat|}}),\ \Gamma_{\anovat}=\Gamma_{t_1}\times\cdots\times\Gamma_{t_{|\anovat|}}, \ |\anovat|>0.
\end{equation}
For a given cardinality $k=0,1,\ldots,\norv$, we define 
\begin{equation}\notag
	\anovaT_{k} := \{\anovat|{\anovat\subseteq \anovau, |\anovat|= k}\},\
	\anovaT_{k}^* :=\cup_{i=1,\ldots,k}\anovaT_i.
\end{equation}
The representation of $u(\pv,\sv)$ in a form
\begin{equation}\label{eq:anova1}
\begin{split}
u(\pv,\sv) &=  u_0(\pv) + \sum_{\anovat\in \anovaT_{\norv}^*}u_{\anovat}(\pv,\sv_{\anovat})\\
&= u_0(\pv) + \sum_{\anovat\in \anovaT_1}u_{\anovat}(\pv,\sv_{\anovat})+\ldots+\sum_{\anovat\in \anovaT_{\norv}}u_{\anovat}(\pv,\sv_{\anovat}),
\end{split}
\end{equation}
is a called an ANOVA decomposition if
\begin{eqnarray}
&&	u_0(\pv) = \int_{\Gamma} \pdf(\sv) u(\pv,\sv)\rd \sv, \label{eq:anova_prop1} \\
&&	\int_{\Gamma_{t_k}} \pdf_{t_k}(\svc_{t_k}) u_{\anovat}(\pv,\sv_{\anovat})\rd\svc_{t_k} = 0, \ t_k \in \anovat, \ |\anovat|>0.\label{eq:anova_prop2}
\end{eqnarray}
We call $u_{\anovat}(\pv,\sv_{\anovat})$ in \eqref{eq:anova1} the $|\anovat|$-th order term or $|\anovat|$-th order component function, and call~$u_0(\pv)$ the $0$-th order term or  $0$-th order component function for special case.

In this work, we assume that the components of the random vector $\sv$ are independent. It follows from~\eqref{eq:anova_prop2} that the terms in \eqref{eq:anova1}  can be expressed as integrals of $u(\pv,\sv)$. To illustrate this, we first show that if $\anovam\nsubseteq \anovat$, then 
\begin{equation}\label{eq:anova_termscompte}
\int_{\Gamma_{\anovat^c}}\pdf_{\anovat^c}(\sv_{\anovat^c})u_{\anovam}(\pv,\sv_{\anovam})\rd \sv_{\anovat^c} = 0,
\end{equation}
where $\anovat^c$ represents the complementary set of $\anovat$, i.e., $\anovat^c=\anovau\backslash\anovat$, and the universal set is given by $\anovau=\{1,\ldots,\norv\}$. In the rest of this paragraph, we prove \eqref{eq:anova_termscompte}. Since  $\anovam\nsubseteq \anovat$, there exists an element  $m_k\in \anovam$ such that $m_k\notin \anovat$; or in other words,  there exists an element  $m_k\in \anovam$ such that $m_k\in \anovat^c$. 
Letting $\anovap = \anovat^c\backslash  \{m_k\}$,  we then have
\begin{equation}\notag
\begin{split}
\int_{\Gamma_{\anovat^c}}\pdf_{\anovat^c}(\sv_{\anovat^c})u_{\anovam}(\pv,\sv_{\anovam})\rd \sv_{\anovat^c} & = \int_{\Gamma_{\anovap}}\int_{\Gamma_{m_k}} \pdf_{m_k}(\svc_{m_k})\pdf_{\anovap}(\sv_{\anovap})u_{\anovam}(\pv,\sv_{\anovam})\rd\svc_{m_k}\rd \sv_{\anovap}\\
&= \int_{\Gamma_{\anovap}} \left(\pdf_{\anovap}(\sv_{\anovap})\int_{\Gamma_{m_k}} \pdf_{m_k}(\svc_{m_k})u_{\anovam}(\pv,\sv_{\anovam})\rd\svc_{m_k}\right)\rd \sv_{\anovap}.
\end{split}
\end{equation}
According to \eqref{eq:anova_prop2}, we have 
\begin{equation}\notag
\int_{\Gamma_{m_k}}\pdf_{m_k}(\svc_{m_k}) u(\pv,\sv_{\anovam})\rd\svc_{m_k} = 0,
\end{equation}
which gives \eqref{eq:anova_termscompte}.

If $\anovam \subseteq \anovat$, then $\anovam\cap \anovat^c = \emptyset$, and we have
\begin{equation}\label{eq:anova_termscompte2}
\int_{\Gamma_{\anovat^c}}\pdf_{\anovat^c}(\sv_{\anovat^c})u_{\anovam}(\pv,\sv_{\anovam})\rd \sv_{\anovat^c} = u_{\anovam}(\pv,\sv_{\anovam})\int_{\Gamma_{\anovat^c}}\pdf_{\anovat^c}(\sv_{\anovat^c})\rd \sv_{\anovat^c}
= u_{\anovam}(\pv,\sv_{\anovam}).
\end{equation}
By using \eqref{eq:anova_termscompte} and \eqref{eq:anova_termscompte2}, we obtain 
\begin{equation}\notag
\int_{\Gamma_{\anovat^c}}\pdf_{\anovat^c}(\sv_{\anovat^c})u(\pv,\sv)\rd \sv_{\anovat^c} = u_{\anovat}(\pv,\sv_{\anovat}) + \sum_{{\anovam}\subset\anovat}u_{\anovam}(\pv,\sv_{\anovam}).
\end{equation}
This formula provides a means to compute the ANOVA terms, as described~\cite{Ma2010adaptive,Yang2012adaptive}:
\begin{equation}\notag
u_{\anovat}(\pv,\sv_{\anovat}) = \int_{\Gamma_{\anovat^c}}\pdf_{\anovat^c}(\sv_{\anovat^c})u(\pv,\sv)\rd \sv_{\anovat^c} - \sum_{{\anovam}\subset\anovat}u_{\anovam}(\pv,\sv_{\anovam}).
\end{equation}

An important property of the ANOVA decomposition is that all the terms in \eqref{eq:anova1} are orthogonal, as follows from \eqref{eq:anova_prop2}. 	To illustrate this, let us assume that $\anovat\neq\anovam$, which implies the existence of an element $t_k\in\anovat$ such that $t_k\notin\anovam$ (if this is not the case, then there exist an element  $m_k\in\anovam$ such that $m_k\notin\anovat$, and the proof follows a similar line of reasoning). let $\anovas= \anovau\backslash\{t_k\}$ be the complementary set of $\{t_k\}$, then we have
\begin{equation}\notag
\begin{split}
\int_{\Gamma}\pdf(\sv)u(\pv,\sv_{\anovat})u(\pv,\sv_{\anovam})\rd \sv &= \int_{\Gamma_{\anovas}}\int_{\Gamma_{t_k}}\pdf_{t_k}(\svc_{t_k})\pdf_{\anovas}(\sv_{\anovas})u(\pv,\sv_{\anovat})u(\pv,\sv_{\anovam})\rd \svc_{t_k}\rd\sv_{\anovas}\\
&=\int_{\Gamma_{\anovas}}\left(\pdf_{\anovas}(\sv_{\anovas})u(\pv,\sv_{\anovam})\int_{\Gamma_{t_k}}\pdf_{t_k}(\svc_{t_k}) u(\pv,\sv_{\anovat})\rd\svc_{t_k}\right)\rd \sv_{\anovas}.
\end{split}
\end{equation}
According to \eqref{eq:anova_prop2}, we have 
\begin{equation}\notag
\int_{\Gamma_{t_k}}\pdf_{t_k}(\svc_{t_k}) u(\pv,\sv_{\anovat})\rd\svc_{t_k} = 0,
\end{equation}
which implies
\begin{equation}\notag
\int_{\Gamma}\pdf(\sv)u(\pv,\sv_{\anovat})u(\pv,\sv_{\anovam})\rd \sv = 0.
\end{equation}

Due to the orthogonality of ANOVA terms, the variance of $u(\pv,\sv)$ is the summation of the variances of all the decomposition terms:}
\begin{equation}\label{eq:anoavvar1}
\var{u} = \sum_{k=1}^{\norv}\sum_{\anovat\in\anovaT_k}\var{u_{\anovat}(\pv,\sv_{\anovat})},
\end{equation}
where
\begin{equation}\label{eq:anovavar2}
\var{u_{\anovat}(\pv,\sv_{\anovat})} = \int_{\Gamma}\pdf(\sv) u_{\anovat}^2(\pv,\sv_{\anovat})\rd\sv=\int_{\Gamma_{\anovat}} \pdf_{\anovat}(\sv_{\anovat}) u_{\anovat}^2(\pv,\sv_{\anovat})\rd\sv_{\anovat}.
\end{equation}

\subsection{Adaptive ANOVA decomposition}
Note that the $k$-th order term in \eqref{eq:anova1}, i.e., $\sum_{\anovat\in \anovaT_{k}}u_{\anovat}(\pv,\sv_{\anovat})$,  has $\binom{\norv}{k}$ terms. For high-dimensional problems, the total number of terms in \eqref{eq:anova1} can be prohibitively large. This motivates the development of an adaptive ANOVA expansion for such problems.  The adaptive ANOVA approach is expected to be a more efficient way to approximate the exact solution since only part of low order terms in \eqref{eq:anova1} is activated based on certain criteria \cite{Yang2012adaptive}.

To determine which terms to include in the ANOVA decomposition, we define sensitivity indices for each term as follows:
\begin{equation}\notag
\mathcal{S}_{\anovat} = \frac{\normp{\var{u_{\anovat}}}}{\sum_{\anovat\in\anovaT_{\norv}^*}\normp{\var{u_{\anovat}}}},
\end{equation}
where $\normp{\cdot}$ denotes the $L^2$ function norm. It follows from equations \eqref{eq:anoavvar1}--\eqref{eq:anovavar2} that
\begin{equation}\notag
0\leq\mathcal{S}_{\anovat}\leq 1,\ \sum_{{\anovat}\in\anovaT_{\norv}^*}\mathcal{S}_{\anovat} = 1.
\end{equation}

The  intuitive way to select the important terms in the ANOVA decomposition \eqref{eq:anova1} is that find the terms such that 
$\mathcal{S}_{\anovat} \geq \tol,$
where $\tol$ is a given tolerance. However, computing all possible terms is computationally expensive since $u(\pv,\sv)$ is the solution of a PDE with random inputs. Instead, we construct the higher order component functions based on the lower order terms in the following way. 

Let $\mathfrak{J}_k\subseteq\anovaT_k$ denote the sets of active indices for each order. Using these active indices, the solution $u(\pv,\sv)$ can be approximated by
\begin{equation}\notag
u(\pv,\sv)\approx 
u_0(\pv) + \sum_{\anovat\in \mathfrak{J}_1}u_{\anovat}(\pv,\sv_{\anovat})+\sum_{\anovat\in \mathfrak{J}_{2}}u_{\anovat}(\pv,\sv_{\anovat}) +\ldots.
\end{equation}
For the first order terms, all the terms are retained, i.e., $\mathfrak{J}_1=\anovaT_1$. Suppose that 
$\mathfrak{J}_k$ is given for  $k\leq \norv - 1$, and define the relative variance $\gamma_{\anovat}$ as
\begin{equation}\label{eq:relativevariance}
\gamma_{\anovat}:= \frac{\normp{\var{u_{\anovat}}}}{\sum_{\anovat\in\mathfrak{J}_k^*}\normp{\var{u_{\anovat}}}}, \ \anovat\in\mathfrak{J}_k^*,
\end{equation}
where $\mathfrak{J}_k^*:= \mathfrak{J}_1\cup\cdots\cup\mathfrak{J}_k$.
Then, the index set of the next order can be  constructed via
\begin{equation}\notag
\mathfrak{J}_{k+1}: = \{\anovat|\anovat\in\anovaT_{k+1}, \mbox{ and } \forall~\anovas\subseteq\anovat \mbox{ with } |\anovas|=k \mbox{ satisfies } \anovas\in\tilde{\mathfrak{J}}_k\}, 
\end{equation}
where 
\begin{equation}\notag
\tilde{\mathfrak{J}}_k:=\{\anovat\in\mathfrak{J}_k|\gamma_{\anovat}\geq\tol\}.
\end{equation}

\begin{algorithm}[ht]
	\caption{Adaptive ANOVA decomposition~\cite{Yang2012adaptive}}\label{alg:adaptive_anova}
	\begin{algorithmic}
		{\normalsize
			\STATE {\bfseries Input:} $u(\pv,\sv)$ and $\tol$, set $k=1$ and $\mathfrak{J}_1 = \{\{1\},\ldots,\{\norv\}\}.$
			\WHILE{$(k<N)$ and $\mathfrak{J}_k\neq \emptyset$}
			\STATE Compute $\gamma_{\anovat}$ for $\anovat\in \mathfrak{J}_k$.
			\STATE Set $\tilde{\mathfrak{J}}_k:=\{\anovat\in\mathfrak{J}_k|\gamma_{\anovat}\geq\tol\}.$
			\STATE Set $\mathfrak{J}_{k+1}: = \{\anovat|\anovat\in\anovaT_{k+1}, \mbox{ and } \forall~\anovas\subseteq\anovat \mbox{ with } |\anovas|=k \mbox{ satisfies } \anovas\in\tilde{\mathfrak{J}}_k\}.$
			\STATE $k=k+1$.
			\ENDWHILE
		}
	\end{algorithmic}
\end{algorithm}
Algorithm~\ref{alg:adaptive_anova} presents the pseudo-code for the adaptive ANOVA decomposition. However, computing the relative variance $\gamma_{\anovat}$ efficiently by classical methods can be challenging, especially when dealing with the high-dimensional random inputs that arise in the context of PDEs with random inputs. To address this issue, we propose an adaptive ANOVA stochastic Galerkin method. In the following sections, we provide details of how to compute the relative variance $\gamma_{\anovat}$ in the adaptive ANOVA stochastic Galerkin method.
\section{Adaptive ANOVA stochastic Galerkin method}\label{sec:Adaptive SGANOVA}
In the last section, we introduced the ANOVA decomposition as a method to capture the important features of the solution.  By representing the component functions of the ANOVA decomposition as the generalized polynomial chaos expansion, an effective surrogate model for the problem~\eqref{eq:spde} can be constructed, which is essential for accelerating the solution evaluation process for time intensive problems. There are two widely used approaches for this purpose: the generalized polynomial chaos (gPC) expansion~\cite{Xiu2002modeling}, and the polynomial dimensional decomposition (PDD) \cite{Tang2016}. In this work, we apply the gPC expansion to present the component functions in \eqref{eq:anova1}. However, it is noteworthy that the PDD can also be used  similarly.

\subsection{Generalized polynomial chaos expansion of component functions}
Let us commence with the definition of gPC basis functions for a single random variable.
Suppose that $\svc_k$ is a random variable with probability density function $\pdfc_k(\svc_k)$, where  $k=1,\ldots,N$. The gPC basis functions are the orthogonal polynomials satisfying
\begin{equation}\label{eq:gpcbasis1}
\int_{\Gamma_{t_k}} \pdfc_k(\svc_k)\gpc_{i}^{(k)}(\svc_k)\gpc_{j}^{(k)}(\svc_k)\rd\svc_k = \delta_{i,j},
\end{equation}
where $i$ and $j$ are non-negative integers, and $\delta_{i,j}$ is the Kronecker delta. 

For $\norv$ dimensional random variables, let $\bm{i} = (i_1,\ldots,i_{\norv})\in\dsN^{\norv}$ be a multi-index with the total degree $|\bm{i}|=i_1+\cdots+i_{\norv}$.  Note that in this work, we assume that the components of $\sv$, i.e., $\svc_1,\ldots,\svc_{\norv}$ are mutually independent, and thus the $\norv$-variate gPC basis functions are the products of the univariate gPC polynomials, i.e.,
\begin{equation}\notag
\GPC_{\bm{i}}(\sv) = \gpc_{i_1}^{(1)}(\svc_1)\cdots\gpc^{(N)}_{i_{\norv}}(\svc_{\norv}).
\end{equation}
It follows from \eqref{eq:gpcbasis1} that 
\begin{equation}\notag
\int_{\Gamma} \pdf(\sv)\GPC_{\bm{i}}(\sv)\GPC_{\bm{j}}(\sv)\rd\sv = \delta_{\bm{i},\bm{j}},
\end{equation}
where $\delta_{\bm{i},\bm{j}} = \delta_{i_1,j_1}\cdots\delta_{i_{\norv},j_{\norv}}$.

We now consider the generalized polynomial chaos expansion of the $|\anovat|$-th order component function $u_{\anovat}(\pv,\sv_{\anovat})$, which requires some notations. Let $\anovat^c=\{t^{c}_1,\ldots,t^{c}_{|\anovat^c|}\}$ be the complementary set of $\anovat$, i.e., \guanjie{$\anovat^c=\anovau\backslash\anovat$}, where the universal set is given by $\anovau=\{1,\ldots,\norv\}$. For any set $\anovat\subseteq \anovau$ with $|\anovat|>0$, the gPC basis function corresponding to the multi-index $\bm{i}_{\anovat}$ is given by  
\begin{equation}\notag
\GPC_{\bm{i}_{\anovat}}(\sv_{\anovat}) = \gpc_{\bm{i}(t_1)}^{(t_1)}(\svc_{t_1})\cdots\gpc_{\bm{i}(t_{|\anovat|})}^{(t_{|\anovat|})}(\svc_{t_{|\anovat|}}), 
\end{equation}
where $\bm{i}_{\anovat}$ denote the multi-index that contains the components of the  multi-index $\bm{i}$ indexed by~$\anovat$, i.e., $\bm{i}_{\anovat} = (\bm{i}({t_1}),\ldots,\bm{i}(t_{|\anovat_k|}) )$. 
In additional, let $\mathfrak{M}_{\anovat}$ be the set of multi-indices defined by
\begin{equation}\notag
\mathfrak{M}_{\anovat}:= \{\bm{i}|\bm{i}\in\dsN^N,\bm{i}(t_1)\neq0,\ldots,\bm{i}(t_{|\anovat|})\neq 0,\bm{i}(t^c_1)= 0,\ldots, \bm{i}(t^c_{|\anovat^c|})= 0\}, \ 0<|\anovat|<N.
\end{equation}
For the special case $|\anovat|= 0$, we set $\mathfrak{M}_{\anovat}=\{\bm{i}|\bm{i}\in \dsN^{\norv}, \bm{i}(1)=0,\ldots,\bm{i}(N)= 0\}$, and for the special case $|\anovat|= N$, we set $\mathfrak{M}_{\anovat}:= \{\bm{i}|\bm{i}\in\dsN^N,\bm{i}(1)\neq0,\ldots,\bm{i}(N)\neq 0\}$. We can then state the following theorem:
\begin{theorem}
Given $\pv\in D$ and $\anovat\subseteq \anovau$ with $|\anovat|>0$, assuming that the  $|\anovat|$-th order component function $u_{\anovat}(\pv,\sv_{\anovat})$ belongs to $L^2_{\pdf}(\Gamma)$, then the generalized polynomial chaos expansion of $u_{\anovat}(\pv,\sv_{\anovat})$ can be expressed by 
\begin{equation}\label{eq:gpcexp1}
u_{\anovat}(\pv,\sv_{\anovat}) = \sum_{\bm{i}\in \mathfrak{M}_{\anovat}}u_{\bm{i}}(\pv)\GPC_{\bm{i}}(\sv),
\end{equation}
where  $u_{\bm{i}}(\pv)$ is the coefficient of $\GPC_{\bm{i}}(\sv)$ defined by 
\begin{equation}\notag
\begin{split}
u_{\bm{i}}(\pv)	&= \int_{\Gamma} \pdf(\sv) u_{\anovat}(\pv,\sv_{\anovat})\GPC_{\bm{i}}(\sv)\rd \sv.
\end{split} 
\end{equation}
\end{theorem}

\begin{proof}
Since $u_{\anovat}(\pv,\sv_{\anovat})\in L^2_{\pdf}(\Gamma)$ and $\{\GPC_{|\bm{i}|}\}_{|\bm{i}|=0}^{\infty}$ forms an complete  orthonormal basis of  $L^2_{\pdf}(\Gamma)$, 
we can express the generalized polynomial chaos expansion of $u_{\anovat}(\pv,\sv_{\anovat})$ as
\begin{equation}\label{eq:gpcexp_original}
u_{\anovat}(\pv,\sv_{\anovat}) = \sum_{|\bm{i}|=0}^{\infty}u_{\bm{i}}(\pv)\GPC_{\bm{i}}(\sv),
\end{equation}
where $u_{\bm{i}}(\pv)$ are the coefficients of the expansion given by
\begin{equation}\notag
\begin{split}
u_{\bm{i}}(\pv)	&= \int_{\Gamma} \pdf(\sv) u_{\anovat}(\pv,\sv_{\anovat})\GPC_{\bm{i}}(\sv)\rd \sv.
\end{split} 
\end{equation}
To prove the theorem, it suffices to show that if $\bm{i}\notin \mathfrak{M}_{\anovat}$, which means that there exists $t_k\in \anovat$ such that $\bm{i}(t_k)=0$ or there exists $t_k^c\in \anovat^c$ such that $\bm{i}(t_k^c)\neq0$, then $u_{\bm{i}}(\pv)=0$.

If there exists $t_k\in \anovat$ such that $\bm{i}(t_k)=0$, \guanjie{let $\anovas= \anovau\backslash\{t_k\}$ be the complementary set of $\{t_k\}$}, and note that $\gpc_0^{(t_k)}(\svc_{t_k})=1$. Then, we have
\begin{equation}\notag
\begin{split}
\int_{\Gamma} \pdf(\sv)u_{\anovat}(\pv,\sv_{\anovat}) \GPC_{\bm{i}}(\sv)\rd \sv
&= \int_{\Gamma_{\anovas}}\int_{\Gamma_{t_k}} \pdf_{\anovas}(\sv_{\anovas})\pdf_{t_k}(\svc_{t_k})\GPC_{\bm{i}_{\anovas}}(\sv_{\anovas})u_{\anovat}(\pv,\sv_{\anovat})\rd \svc_{t_k} \rd\sv_{\anovas}\\
&= \int_{\Gamma_{\anovas}} \left( \pdf_{\anovas}(\sv_{\anovas})\GPC_{\bm{i}_{\anovas}}(\sv_{\anovas})\int_{\Gamma_{t_k}}\pdf_{t_k}(\svc_{t_k})u_{\anovat}(\pv,\sv_{\anovat})\rd \svc_{t_k} \right)\rd\sv_{\anovas}.
\end{split}
\end{equation}
Using \eqref{eq:anova_prop2}, we obtain
\begin{equation}\notag
\begin{split}
\int_{\Gamma_{t_k}} \pdf_{t_k}(\svc_{t_k}) u_{\anovat}(\pv,\sv_{\anovat})\rd\svc_{t_k}  =0,
\end{split}
\end{equation}
which implies that
\begin{equation}\notag
\begin{split}
u_{\bm{i}}(\pv)	&= \int_{\Gamma} \pdf(\sv) u_{\anovat}(\pv,\sv_{\anovat})\GPC_{\bm{i}}(\sv)\rd \sv=0.
\end{split} 
\end{equation}

On the other hand, if there exists $t_k^c\in \anovat^c$ such that $\bm{i}(t_k^c)\neq0$, then we have
\begin{equation}\notag
\begin{split}
\int_{\Gamma} \pdf(\sv) u_{\anovat}(\pv,\sv_{\anovat})\GPC_{\bm{i}}(\sv)\rd \sv 
&= \int_{\Gamma_{\anovat}}\int_{\Gamma_{\anovat^c}} \pdf_{\anovat}(\sv_{\anovat})\pdf_{\anovat^c}(\sv_{\anovat^c}) u_{\anovat}(\pv,\sv_{\anovat}) \GPC_{\bm{i}_{\anovat}}(\sv_{\anovat})\GPC_{\bm{i}_{\anovat^c}}(\sv_{\anovat^c})\rd \sv_{\anovat}\rd \sv_{\anovat^c}\\
&=\int_{\Gamma_{\anovat}} \pdf_{\anovat}(\sv_{\anovat})u_{\anovat}(\pv,\sv_{\anovat}) \GPC_{\bm{i}_{\anovat}}(\sv_{\anovat})\rd \sv_{\anovat} \int_{\Gamma_{\anovat^c}} \pdf_{\anovat^c}(\sv_{\anovat^c}) \GPC_{\bm{i}_{\anovat^c}}(\sv_{\anovat^c})\rd \sv_{\anovat^c}.
\end{split} 
\end{equation}
Note that since the gPC basis function corresponding to $(0,\ldots,0)$ is $1$, and $\bm{i}_{\anovat^c}\neq (0,\ldots,0)$, we have
\begin{equation}\notag
\int_{\Gamma_{\anovat^c}} \pdf_{\anovat^c}(\sv_{\anovat^c}) \GPC_{\bm{i}_{\anovat^c}}(\sv_{\anovat^c})\rd \sv_{\anovat^c}=0,
\end{equation}
and thus 
\begin{equation}\notag
\begin{split}
u_{\bm{i}}(\pv)	&= \int_{\Gamma} \pdf(\sv) u_{\anovat}(\pv,\sv_{\anovat})\GPC_{\bm{i}}(\sv)\rd \sv=0.
\end{split} 
\end{equation}
\qed
\end{proof}

\begin{theorem}
Suppose that $u(\pv,\sv)$ can be expressed as 
\begin{equation}\label{eq:gpcanova_thm2}
u(\pv,\sv)= u_0(\pv) + \sum_{\anovat\in \anovaT_1}u_{\anovat}(\pv,\sv_{\anovat})+\ldots+\sum_{\anovat\in \anovaT_{\norv}}u_{\anovat}(\pv,\sv_{\anovat}),
\end{equation}
where 
\begin{equation}\label{eq:them2b}
u_{\anovat}(\pv,\sv_{\anovat})=\sum_{\bm{i}\in\mathfrak{M}_{\anovat}}u_{\bm{i}}(\pv)\GPC_{\bm{i}}(\sv), \ \anovat\in \anovaT_{\norv}^*.
\end{equation}
Then the right hand side of \eqref{eq:gpcanova_thm2} is the ANOVA decomposition of $u(\pv,\sv)$.
\end{theorem}

\begin{proof}
To complete the proof, we only need to show that the right hand side of  \eqref{eq:gpcanova_thm2} satisfies~\eqref{eq:anova_prop1} and \eqref{eq:anova_prop2}. Using \eqref{eq:gpcanova_thm2} and \eqref{eq:them2b}, we have 
\begin{equation}\notag
\begin{split}
 \int_{\Gamma}\pdf(\sv) u(\pv,\sv)\rd \sv & =\int_{\Gamma}\pdf(\sv) u_0(\pv) \rd\sv + \sum_{\anovat\in \anovaT_N^*}\int_{\Gamma}\pdf(\sv) u_{\anovat}(\pv,\sv_{\anovat})\rd \sv\\
& = u_0(\pv) + \sum_{\anovat\in \anovaT_{\norv}^*}\sum_{\bm{i}\in\mathfrak{M}_{\anovat}}u_{\bm{i}}(\pv)\int_{\Gamma}\pdf(\sv) \GPC_{\bm{i}}(\sv)\rd \sv.
\end{split}
\end{equation}
Since $\anovat\neq \emptyset$ when $\anovat\in \anovaT_{\norv}^*$, we have
\begin{equation}\notag
\bm{i} \neq (0,\ldots,0), \ \bm{i}\in\mathfrak{M}_{\anovat}.
\end{equation}
Thus, 
\begin{equation}\notag
\int_{\Gamma}\pdf(\sv) \GPC_{\bm{i}}(\sv)\rd \sv=0,
\end{equation}
which implies that 
\begin{equation}\notag
\int_{\Gamma}\pdf(\sv) u(\pv,\sv)\rd \sv = u_0(\pv).
\end{equation}
To show that the right hand side of  \eqref{eq:gpcanova_thm2} satisfies \eqref{eq:anova_prop2}, suppose that $t_k\in \anovat$, $k=1,\ldots,t_{|\anovat|}$, \guanjie{and let $\anovas= \anovau\backslash\{t_k\}$ be the complementary set of $\{t_k\}$}. Then, by  \eqref{eq:them2b}, we have 
\begin{equation}\notag
\begin{split}
\int_{\Gamma_{t_k}} \pdf_{t_k}(\svc_{t_k}) u_{\anovat}(\pv,\sv_{\anovat})\rd\svc_{t_k} 
= &\sum_{\bm{i}\in\mathfrak{M}_{\anovat}}u_{\bm{i}}(\pv)\int_{\Gamma_{t_k}}\pdf_{t_k}(\svc_{t_k}) \GPC_{\bm{i}}(\sv)\rd\svc_{t_k}\\
= & \sum_{\bm{i}\in\mathfrak{M}_{\anovat}}u_{\bm{i}}(\pv)\GPC_{\bm{i}_{\anovas}}(\sv_{\anovas})\int_{\Gamma_{t_k}}\pdf_{t_k}(\svc_{t_k}) \gpc_{\bm{i}(t_k)}^{(t_k)}(\svc_{t_k})\rd\svc_{t_k}
\end{split}
\end{equation}
Recall that $\gpc_0^{(t_k)}(\svc_{t_k})=1$ and $\bm{i}(t_k)\neq0$, we have 
\begin{equation}\notag
\int_{\Gamma_{t_k}}\pdf_{t_k}(\svc_{t_k}) \gpc_{\bm{i}(t_k)}^{(t_k)}(\svc_{t_k})\rd\svc_{t_i} = 0,
\end{equation}
which implies  that 
\begin{equation}\notag
\int_{\Gamma_{t_k}} \pdf_{t_k}(\svc_{t_k}) u_{\anovat}(\pv,\sv_{\anovat})\rd\svc_{t_k} = 0,\ k =1,\ldots,t_{|\anovat|}.
\end{equation}
\qed
\end{proof}

In practical computations, the expansion given in~\eqref{eq:gpcexp1} must be truncated to a finite number of terms. Following the approach in \cite{Xiu2010}, 
we retain the terms with the total degree up to $p$. This yields the approximation
\begin{equation}\label{eq:gpcexp2}
u_{\anovat}(\pv,\sv_{\anovat}) \approx \sum_{\bm{i}\in\mathfrak{M}_{\anovat}^p}u_{\bm{i}}(\pv)\GPC_{\bm{i}}(\sv),
\end{equation}
where $\mathfrak{M}_{\anovat}^p$ is the set of multi-indices defined by
\begin{equation}\notag
\mathfrak{M}_{\anovat}^p:=\{\bm{i}| \bm{i}\in\mathfrak{M}_{\anovat}\ \mbox{and}\ |\bm{i}| \leq p \}.
\end{equation}
By inserting \eqref{eq:gpcexp2} into \eqref{eq:anova1}, we obtain the  following expansion of $u(\pv,\sv)$ in terms of gPC basis functions:
\begin{equation}\label{eq:gpcanova}
\begin{split}
u(\pv,\sv)\approx u_p(\pv,\sv) = u_0(\pv) 
+ \sum_{\anovat\in \anovaT_1}\sum_{\bm{i}\in\mathfrak{M}_{\anovat}^p}u_{\bm{i}}(\pv)\GPC_{\bm{i}}(\sv)
+\ldots 
+\sum_{\anovat\in \anovaT_{\norv}}\sum_{\bm{i}\in\mathfrak{M}_{\anovat}^p}u_{\bm{i}}(\pv)\GPC_{\bm{i}}(\sv),
\end{split}
\end{equation}
where $u_p(\pv,\sv)$ is the polynomial approximation of $u(\pv,\sv)$ with the total degree up to $p$. By employing \eqref{eq:gpcexp2}, we can easily compute the variance of $u_{\anovat}(\pv,\sv_{\anovat})$ using the following expression:
\begin{equation}\label{eq:variance1}
\var{u_{\anovat}} \approx \sum_{\bm{i}\in\mathfrak{M}_{\anovat}^p}u_{\bm{i}}^2(\pv),\ \anovat\in\anovaT_{\norv}^* .
\end{equation}
\begin{figure}[!htbp]
	\begin{center}		
		\subfloat[]{\includegraphics[width=0.41\linewidth]{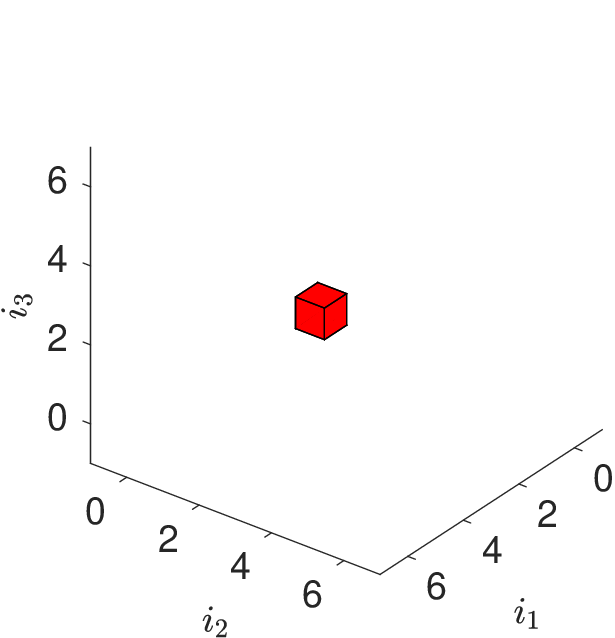}}\qquad
		\subfloat[]{
			\includegraphics[width=0.41\linewidth]{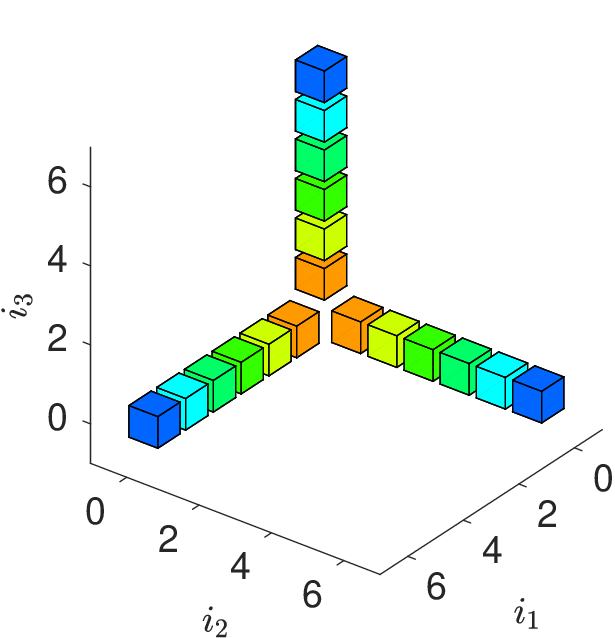}}\qquad
		\subfloat[]{\includegraphics[width=0.41\linewidth]{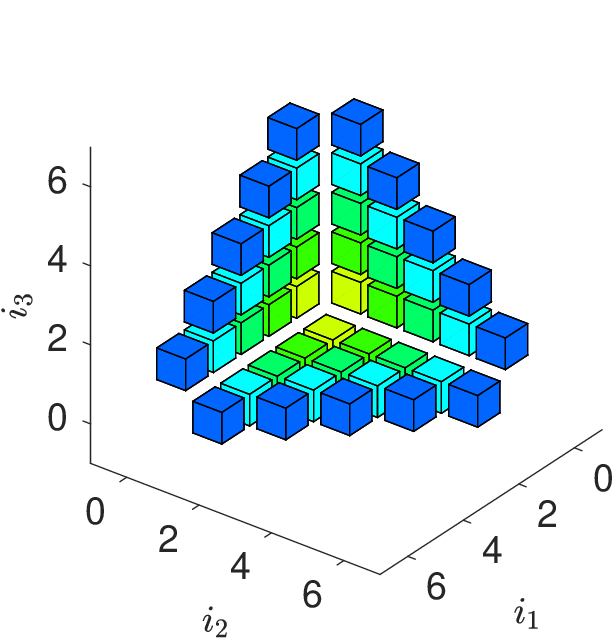}}\qquad
		\subfloat[]{\includegraphics[width=0.41\linewidth]{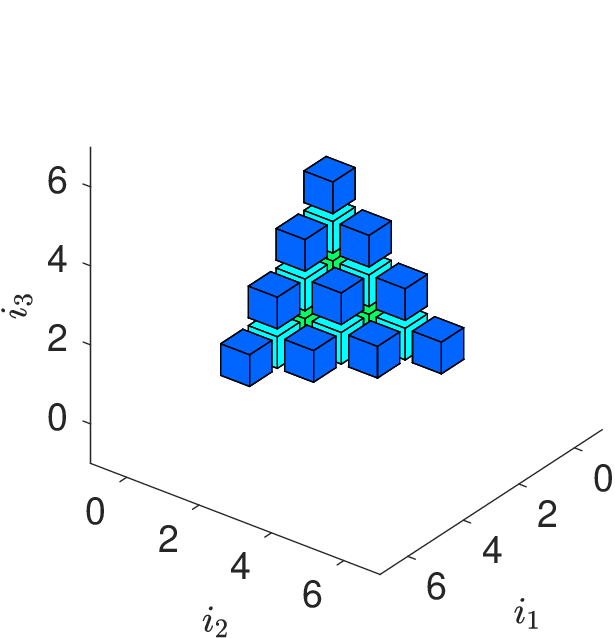}}
		\caption{Multi-indices of the gPC basis functions corresponding to the component functions of each order in $3$ dimensions with the total degree up to $6$, arranged according to the order of the component functions (from left to right): $0$-th, first, second, and third order.}\label{fig:anova}
	\end{center}
\end{figure}%
Fig. \ref{fig:anova} illustrates the multi-indices of the gPC basis functions corresponding to the component functions of each order in~\eqref{eq:gpcanova}. It is worth noting that the number of terms in \eqref{eq:gpcanova} is given by
\begin{equation}\label{eq:Vandermonde}
\binom{N}{0}\binom{p}{0}+\cdots+\binom{N}{N}\binom{p}{N} = \binom{N+p}{N},
\end{equation}
which identical to the number of terms in the generalized polynomial chaos expansion with the total degree up to $p$. The equation \eqref{eq:Vandermonde} is commonly referred to as the Vandermonde's identity or the Vandermonde's convolution. Interested readers can find more information on this topic in the relevant literature, such as \cite{Askey1975orthogonal}.

\subsection{Adaptive ANOVA stochastic Galerkin method}
Based on the adaptive ANOVA decomposition and the gPC expansion of component functions, we can develop an adaptive ANOVA stochastic Galerkin method  for the problem \eqref{eq:spde}. The idea is quite simple, namely,  we select the basis functions of stochastic space based on the adaptive ANOVA decomposition, in which the relative variance is computed by  \eqref{eq:relativevariance} and \eqref{eq:variance1}.

To give the algorithm of this procedure, some notations are needed. We first collect the basis functions associated with the $k$-th order component function $u_{\anovat}(\pv,\sv_{\anovat})$ and denote the set of their multi-indices with the total degree up to $p$ as $\mathfrak{M}_k^p:=\cup_{{\anovat}\in\mathfrak{J}_k}\mathfrak{M}_{\anovat}^p$. 
Moreover, let us denote the set of all multi-indices as $\mathfrak{M}_k^{p^\dagger}:=\mathfrak{M}_{\emptyset}\cup\mathfrak{M}_k^{p^*}$, where $\mathfrak{M}_k^{p^*}:=\cup_{{\anovat}\in\mathfrak{J}_k^*}\mathfrak{M}_{\anovat}^p.$
With those notations, 
Algorithm~\ref{alg:adaptivestoch} gives the pseudocode for the adaptive ANOVA  stochastic Galerkin method.
\begin{algorithm}[ht]
	\caption{Adaptive ANOVA stochastic Galerkin method}\label{alg:adaptivestoch}
	\begin{algorithmic}
		{\normalsize
			\STATE {\bfseries Input:} The gPC order $p$ and the tolerance $\tol$ in ANOVA decomposition.
			\STATE Set $k=1$, $\mathfrak{J}_0=\emptyset,\ \mathfrak{J}_1 = \{\{1\},\ldots,\{\norv\}\}$ and compute $\bm{A}_i$ and $\bm{f}$ defined in \eqref{eq:phy_mat}.
			\WHILE{$(k<N)$ and $\mathfrak{J}_k\neq \emptyset$}
			\STATE Generate the  multi-indices $\mathfrak{M}_{k}^{p^\dagger}$ and compute $\bm{G}_i$ and $\bm{h}$  defined in \eqref{eq:stoch_mat}.
			\STATE Solve the linear system \eqref{eq:linsg} and  compute $\gamma_{\anovat}$ for $\anovat\in \mathfrak{J}_k$ by \eqref{eq:relativevariance} and \eqref{eq:variance1}.
			\STATE Set $\tilde{\mathfrak{J}}_k:=\{\anovat\in\mathfrak{J}_k|\gamma_{\anovat}\geq\tol\}.$
			\STATE Set $\mathfrak{J}_{k+1}: = \{\anovat|\anovat\in\anovaT_{k+1}, \mbox{ and } \forall~\anovas\subseteq\anovat \mbox{ with } |\anovas|=k \mbox{ satisfies } \anovas\in\tilde{\mathfrak{J}}_k\}.$
			\STATE $k=k+1$.
			\ENDWHILE
			\STATE {\bfseries Return:} the approximation $u^{\ap}(\pv,\sv)$, the mean function $u_0(\pv)$ 
			and the variance function $ \sum_{\bm{i}\in\mathfrak{M}_{\anovat}^{p^*}}u_{\bm{i}}^2(\pv)$
		}
	\end{algorithmic}
\end{algorithm}

In the adaptive ANOVA stochastic Galerkin method, we select the set of multi-indices adaptively based on the ANOVA decomposition. Specifically, only part of the multi-indices with the total degree up to $p$ will be retained,  resulting in a much lower computational cost compared to the standard stochastic Galerkin method. It is worth noting that if the tolerance $\tol$ is chosen small enough, all multi-indices will be selected, and the adaptive ANOVA stochastic Galerkin method will become equivalent to the standard stochastic Galerkin method.

\section{Numerical results}\label{sec:numericaltests}
In this section, we will explore two problems: a diffusion problem and a Helmholtz problem. All the results presented here are obtained using MATLAB R2015b on a desktop with a 2.90GHz Intel Core i7-10700 CPU. The CPU time reported in this paper correspond to the total time required to solve the linear systems in the respective procedures. 

To assess the accuracy, we define the mean errors and the variance errors as follows:
\begin{eqnarray*}
	\mathrm{E}_{\err} = \frac{\normp{\mean{u_{\refs}(\pv,\sv)}-\mean{u^{\ap}(\pv,\sv)}}}{\normp{ \mean{u_{\refs}(\pv,\sv)}}},\label{eq:deferr}\\
	\mathrm{V}_{\err}=\frac{\normp{\var{u_{\refs}(\pv,\sv)}-\var{u^{\ap}(\pv,\sv)}}}{\normp{ \var{u_{\refs}(\pv,\sv)}}}.\label{eq:var_err}
\end{eqnarray*}
Here, $u_{\refs}(\pv,\sv)$ is the reference solution, and $u^{\ap}(\pv,\sv)$ is the approximate solution.

\subsection{Test problem 1}
In this problem, we investigate the diffusion equation with random inputs, given by
\begin{equation}\notag
\begin{split}
-\nabla\cdot(a(\pv,\sv)\nabla u(\pv,\sv)) = 1 &\quad \mbox{in}\quad D\times\Gamma,\\
u(\pv,\sv) = 0 &\quad \mbox{on}\quad \partial D\times \Gamma,\\
\end{split}
\end{equation}
where $D = [0,1]\times[0,1]$ is the spatial domain, and $\partial D$ represents the boundary of $D$. The diffusion coefficient $a(\pv,\sv)$ is modeled as a truncated Karhunen--Lo\`eve (KL) expansion~\cite{Ghanem2003,Elman2007} of a random field with a mean function $a_0(\pv)$, a standard deviation $\sigma={1}/{4}$, and the covariance function $\mathrm{Cov}\,(\bm{x},\bm{y})$ given by
\begin{equation}\notag
\mathrm{Cov}\,(\bm{x},\bm{y})=\sigma^2 \exp\left(-\frac{|x_1-y_1|}{c}-\frac{|x_2-y_2|}{c}
\right),\label{covariance}
\end{equation}
where $\pv=[x_1,x_2]^T$, $\bm{y}=[y_1,y_2]^T$ and $c=1/4$ is the correlation length.
The KL expansion takes the form
\begin{equation}\label{eq:kl}
a(\pv,\sv)=a_0(\pv)+\sum_{i=1}^{N}a_i(\pv)\svc_i=a_0(\pv)+\sum_{i=1}^{N}\sqrt{\lambda_i}c_i(\pv)\svc_i,
\end{equation}
where $a_0(\pv)=1$, $\{\lambda_i,c_i(\pv)\}_{i=1}^{N}$ are the eigenpairs of  $\mathrm{Cov}\,(\pv,\bm{y})$,
$\{\svc_i\}^{N}_{i=1}$ are uncorrelated random variables, and
$N$ is the number of KL modes retained. 

\begin{table}[ht]%
	\caption{Parameters of the diffusion coefficient $a(\pv,\sv)$ in \eqref{eq:kl} and $K$ in \eqref{eq:affine1}.}\label{tab:diff1_set}
	\begin{center}%
		\newcolumntype{C}{>{\centering\arraybackslash}X}%
		\begin{tabularx}{0.9\linewidth}{CCCCCC}%
			\toprule
			Case  &   $N$ & $K$ \\%
			\midrule
			I     &   $10$ & $11$\\%
			II    &   $50$ & $51$\\%
			\bottomrule%
		\end{tabularx}%
	\end{center}%
\end{table}%

For this test problem, we assume that the random variables  $\{\svc_i\}^{N}_{i=1}$ are independent and uniformly distributed within the range $[-1,1]$. The parameters of $a(\pv,\sv)$ are set as shown in Table~\ref{tab:diff1_set}. In the physical domain, the meshgrid is set to $33\times33$ (i.e., $\nphy=33$). In the stochastic space, the total degree of gPC in the  adaptive ANOVA stochastic Galerkin (AASG) method is specified as $p=5$. The linear systems arising from both the standard and the adaptive ANOVA stochastic Galerkin methods are solved using the preconditioned conjugate gradients (CG) method, with a tolerance of $10^{-8}$ and the mean based preconditioner~\cite{Powell2009}. 

\guanjie{In this test problem, we compare the the adaptive ANOVA stochastic Galerkin (AASG) method with the anchored ANOVA stochastic collocation (AASC) method \cite{Ma2010adaptive,Yang2012adaptive} and the Monte Carlo method (MCM). For the AASC method, we follow the method described in~\cite{Ma2010adaptive} 
 with the relative mean $\tilde{\gamma}_{\anovat}$, i.e.,
	\begin{equation}\notag
	\tilde{\gamma}_{\anovat}:= \frac{\normp{\mean{u_{\anovat}}}}{\normp{\mean{u_{0}}}}, \ \anovat\in\mathfrak{J}_k^*,
	\end{equation}
	as the criterion for selecting important terms within the ANOVA decomposition. Additionally, the mean value of $\sv$ is used as the anchor point. In the AASC method, we adopt tensor style Gaussian quadrature points with a grid level of 5 as the collocation points, resulting in a total of $6^{|\anovat|}$ collocation (quadrature) points for each $\anovat \in \mathfrak{J}_k^*$. Furthermore, for the AASC method, the variance function is computed following the  method proposed in~\cite{Tang2015sensitivity}. For both the MCM and the AASC method, the linear systems are solved using the MATLAB backslash solver.}

\subsubsection{Case I: a $10$ dimensional diffusion problem}
We consider the AASG method with decreasing tolerances $\tol = \{10^{-1},10^{-3},10^{-5},10^{-7},10^{-9}\}$ to demonstrate its effectiveness and efficiency. To access the accuracy,  we obtain the reference solution $u_{\refs}(\pv,\sv)$ using the standard stochastic Galerkin method with the total degree of up to $p=7$.

\begin{table}[ht]
	\caption{Performance of the AASG method for test problem 1 with $N=10$.}\label{tab:diff1a}
	{\normalsize
		\begin{center}
			\begin{tabular}{lccccccccccccc}
				\toprule
				$\tol$ & $\vert\mathfrak{J}_1\vert$  &$\vert\tilde{\mathfrak{J}}_1\vert$ & $\vert\mathfrak{J}_2\vert$ & $\vert \tilde{\mathfrak{J}}_2\vert$ & $\vert\mathfrak{J}_3\vert$  & $\vert \tilde{\mathfrak{J}}_3\vert$ & $\vert\mathfrak{J}_4\vert$  & $\vert \tilde{\mathfrak{J}}_4\vert$ & $\vert\mathfrak{J}_5\vert$  & $\vert \tilde{\mathfrak{J}}_5\vert$ & $k$ &$\vert{\mathfrak{M}}_k^{5^\dagger}\vert$ & CPU time\\
				\midrule				
				$10^{-1}$  &   10  &    1  &    0  &    0  &    0  &    0  &    0  &    0  &    0  &    0  & 1  &    51  & 0.07\\
				$10^{-3}$  &   10  &   10  &   45  &    2  &    0  &    0  &    0  &    0  &    0  &    0  & 2  &   501  & 0.87\\
				$10^{-5}$  &   10  &   10  &   45  &   37  &   70  &    0  &    0  &    0  &    0  &    0  & 3  &  1201  & 3.35\\
				$10^{-7}$  &   10  &   10  &   45  &   45  &  120  &   75  &   60  &    0  &    0  &    0  & 4  &  2001  & 9.38\\
				$10^{-9}$  &   10  &   10  &   45  &   45  &  120  &  120  &  210  &  127  &   70  &    0  & 5  &  2821  & 18.71\\
				\bottomrule
			\end{tabular}
		\end{center}
	}
\end{table}
Table \ref{tab:diff1a} presents the number of active indices for each order of the ANOVA decomposition and the total number of selected gPC basis functions in the stochastic space. Furthermore, we report the computational time required for solving all linear systems that arise during the execution of the while loop in Algorithm~\ref{alg:adaptivestoch}. It can be observed that the number of selected gPC basis functions increases as the tolerance $\tol$ decreases, and therefore, the accuracy can be improved by reducing $\tol$. For a $10$ dimensional problem, the number of gPC basis functions with the total degree up to $5$ is $C_{15}^5=3003$. From the table, it can be seen that when $\tol=10^{-9}$, almost all the gPC basis functions are selected in the AASG method. Further decreasing the tolerance in the AASG method results in the selection of all gPC basis functions with the total degree up to $5$, making the AASG method equivalent to the standard stochastic Galerkin method.

\begin{figure}[!htbp]
	\begin{center}
		\subfloat[$\tol = 10^{-3}$.]{\includegraphics[height=0.27\linewidth]{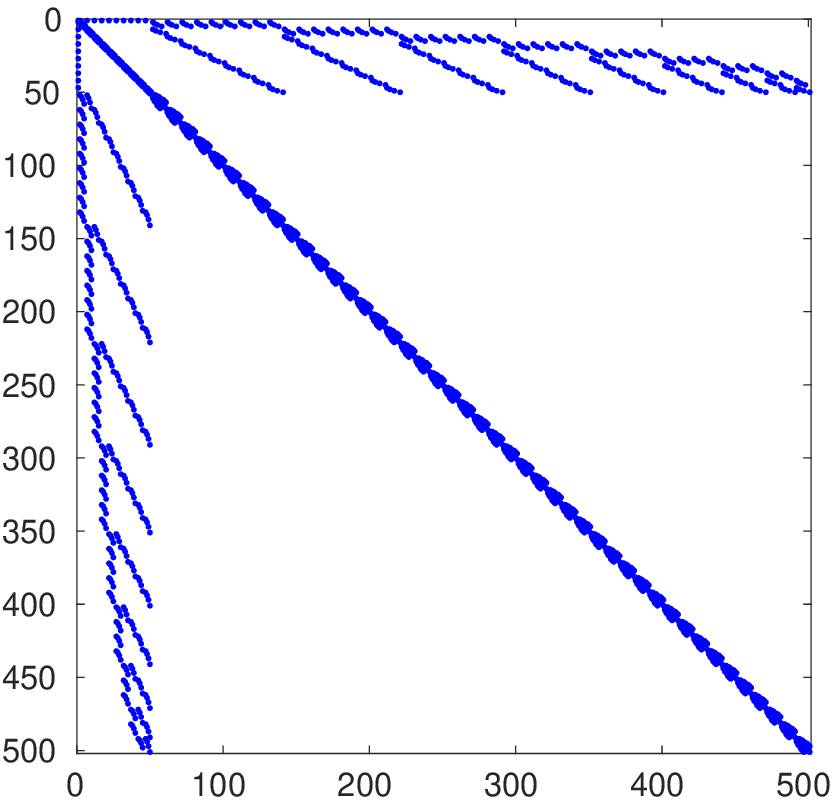}} \qquad\quad
		\subfloat[$\tol = 10^{-5}$.]{\includegraphics[height=0.27\linewidth]{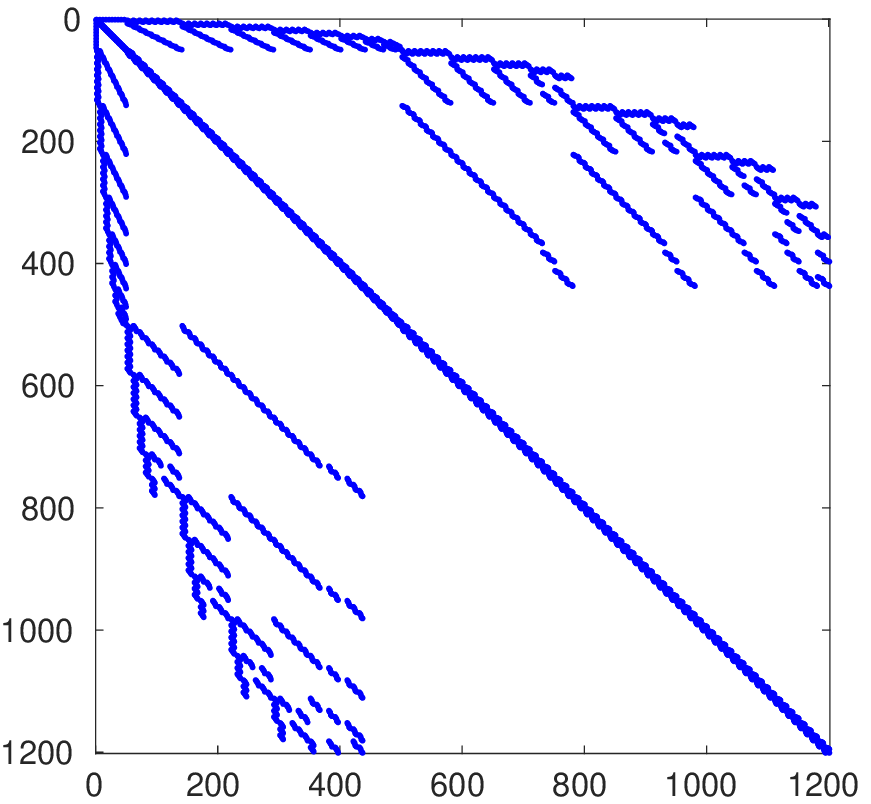}} \qquad\quad
		\subfloat[$\tol = 10^{-7}$.]{\includegraphics[height=0.27\linewidth]{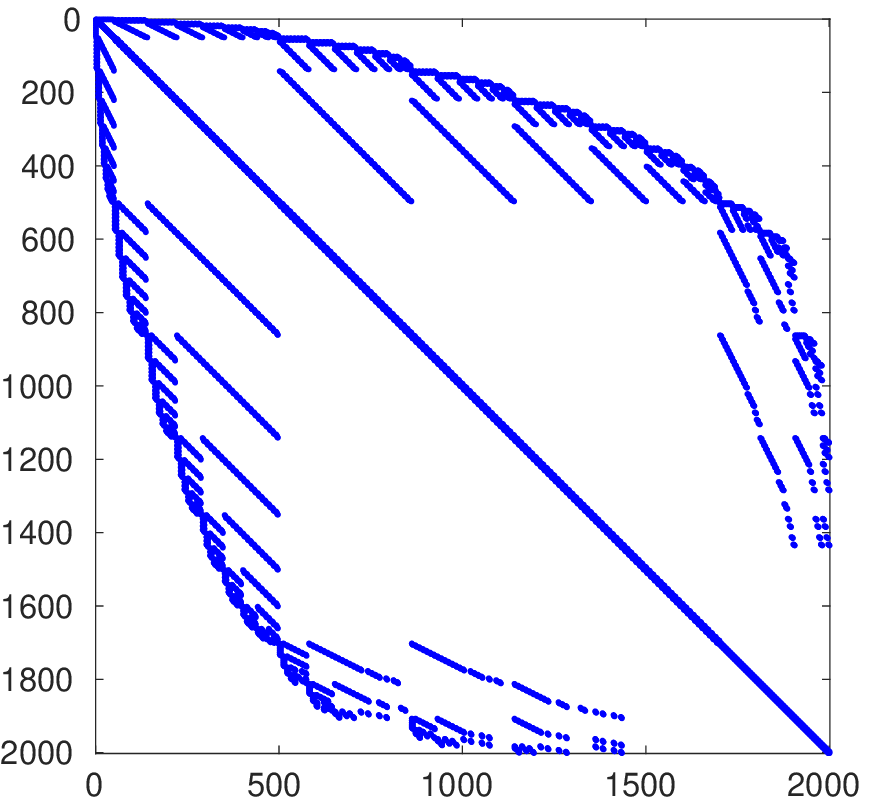}}
		\caption{Matrix block-structure (each block has dimension $\nphy\times\nphy$) for test problem 1 with $N=10$.}\label{fig:diff1a}
	\end{center}
\end{figure}

Fig. \ref{fig:diff1a} displays the block structure of the coefficient matrix of the resulting linear system. Each point in the figure represents a block of dimension $\nphy\times\nphy$. Furthermore, each nonzero block of the coefficient matrix has the same sparsity pattern as the corresponding deterministic problem. Therefore, the coefficient matrix is extremely large and sparse, and the resulting linear system should be solved using iterative methods.

\begin{figure}[!htbp]
	\begin{center}		
		\subfloat[Mean errors w.r.t. CPU time.]{\includegraphics[width=0.46\linewidth]{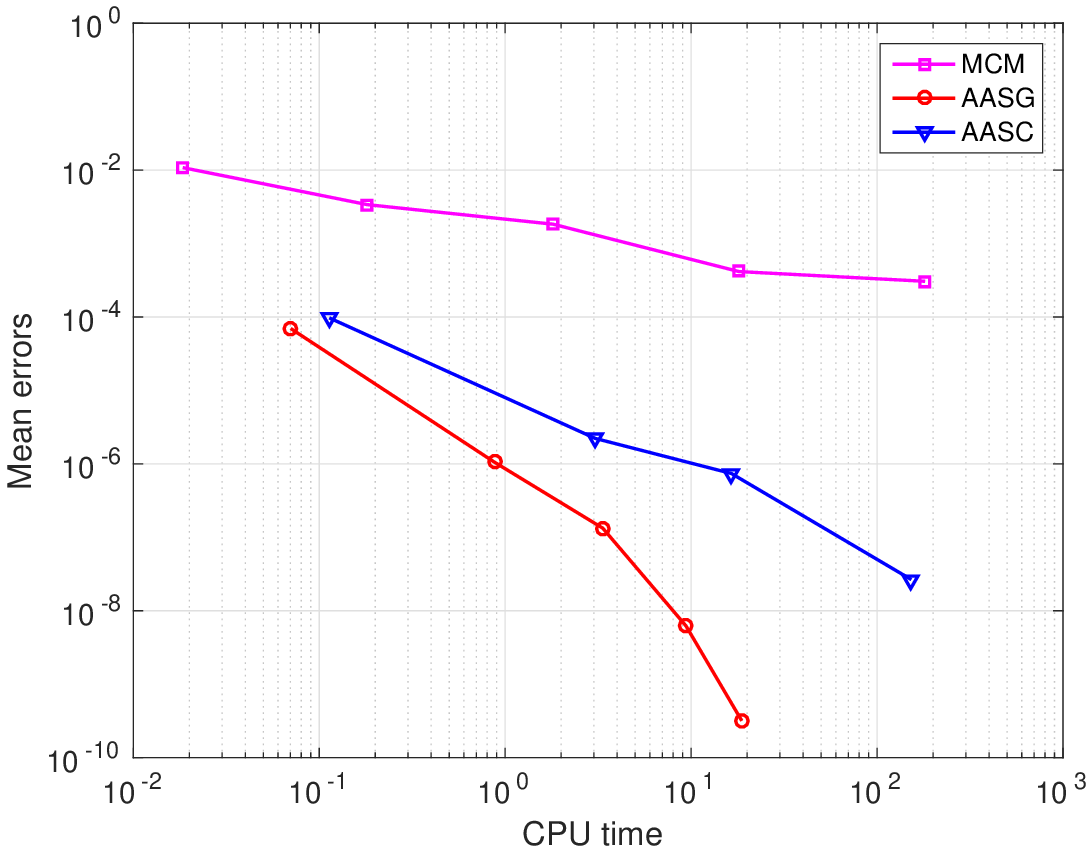}}\quad
		\subfloat[Variance errors w.r.t. CPU time.]{\includegraphics[width=0.46\linewidth]{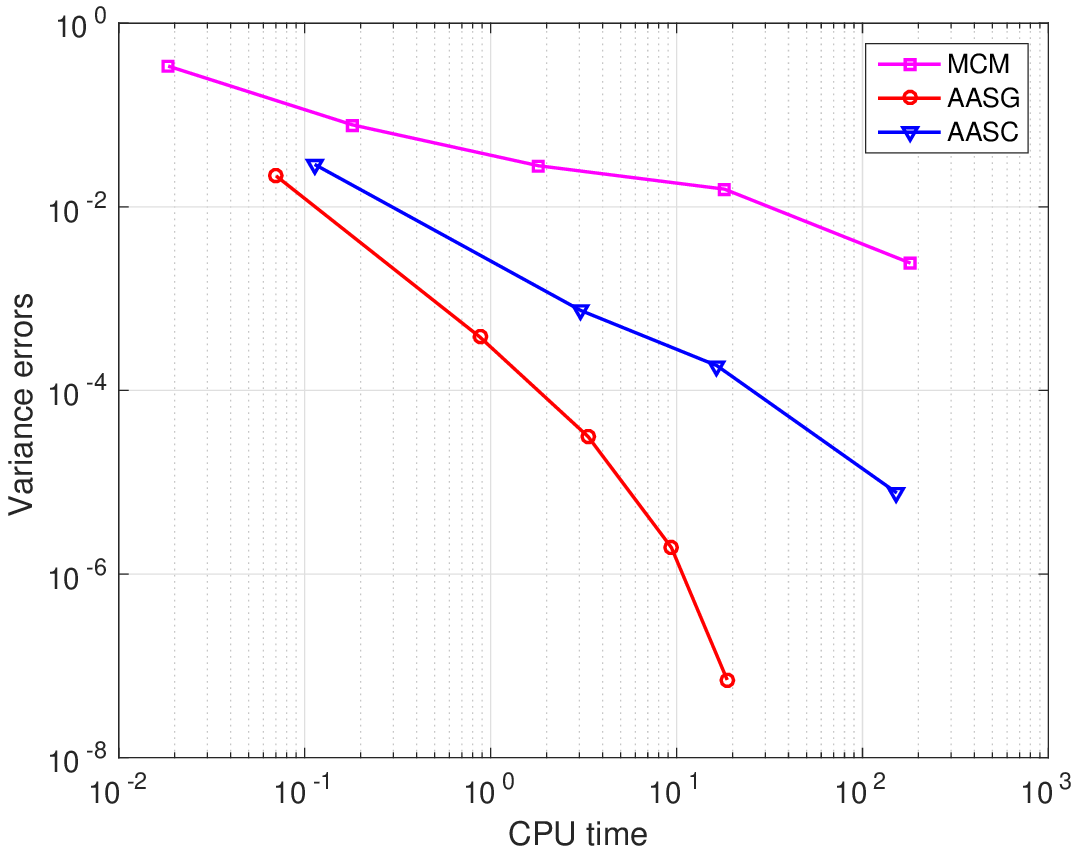}}\\
		\subfloat[Mean errors w.r.t. stochastic DOF.]{\includegraphics[width=0.46\linewidth]{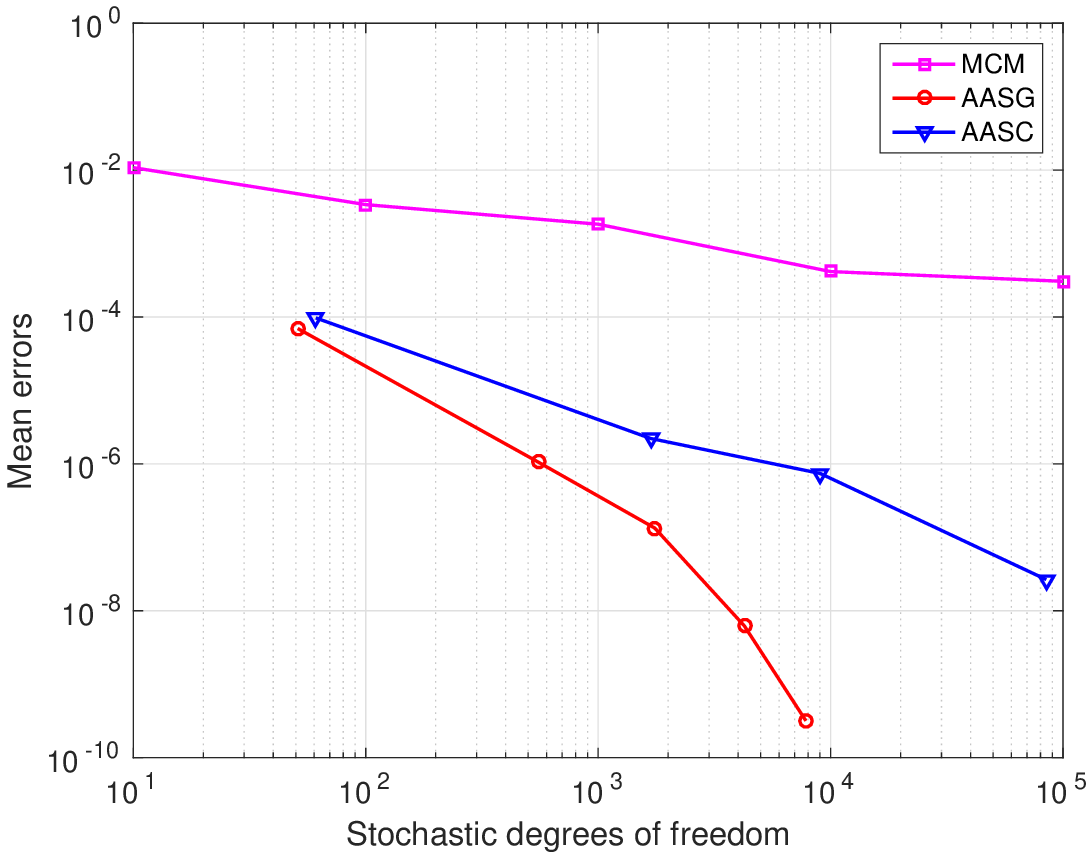}}\quad
		\subfloat[Variance errors w.r.t. stochastic DOF.]{\includegraphics[width=0.46\linewidth]{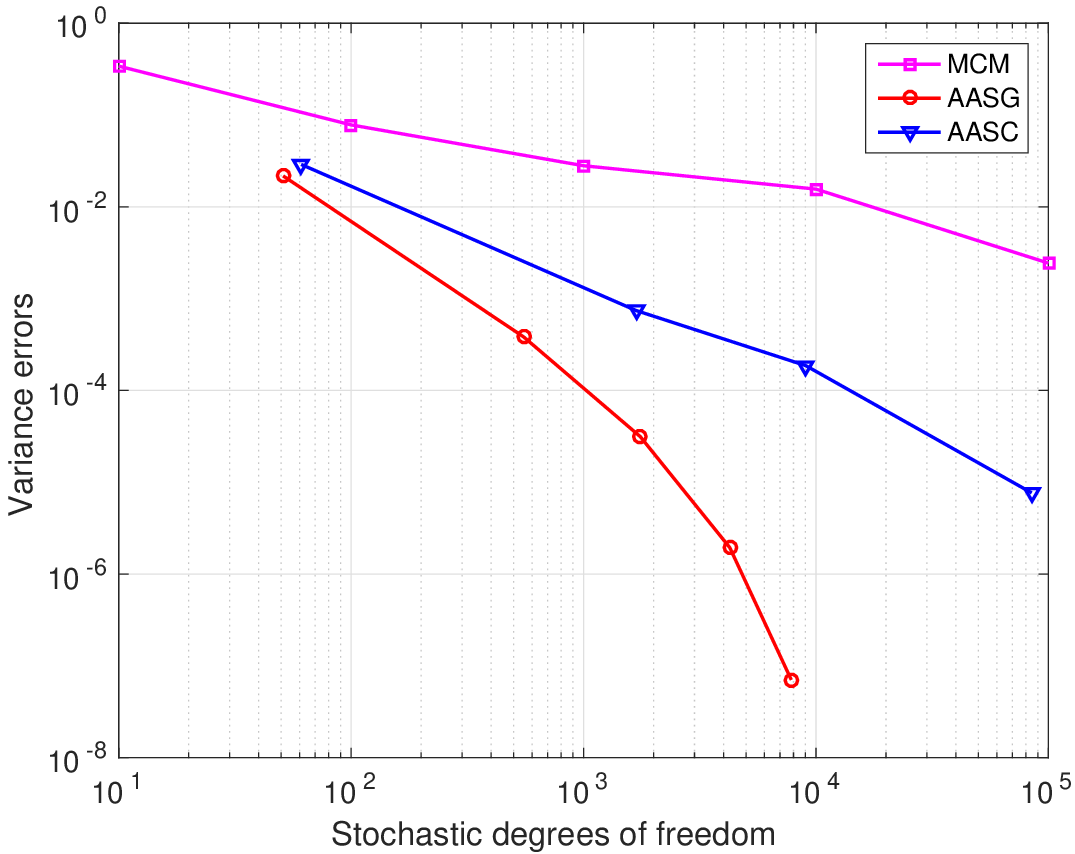}}\\
	\end{center}
	\caption{\guanjie{Comparison of errors with respect to CPU times and stochastic degrees of freedom for test problem 1 with $N=10$, where both the total degree of gPC in the AASG method and the grid level in the AASC method are set to $5$.}}\label{fig:diff1b}
\end{figure}

\guanjie{Fig. \ref{fig:diff1b} investigates the accuracy achieved by the three methods, presenting errors concerning both CPU time and stochastic degrees of freedom (DOF). For clarity, CPU time denotes the total time required to solve all the linear systems within the respective procedures. For the AASG method, stochastic degrees of freedom encompass the cumulative count of gPC basis functions generated during the execution of the while loop in Algorithm~\ref{alg:adaptivestoch}, while for the MCM and the AASC method, it corresponds to the total number of sample points used. The results indicate the notable efficiency of both the AASG method and the AASC method in comparison to the MCM, as they are hundreds of times faster in terms of CPU time and stochastic degrees of freedom. Furthermore, the AASG method outperforms the AASC method in terms of CPU time, and it is also evident that the AASG method provides a higher accuracy per stochastic degree of freedom compared to the AASC method.}

\subsubsection{Case II: a 50 dimensional diffusion problem}
We consider the AASG method with decreasing tolerances $\tol = \{10^{-1},10^{-2},10^{-3},10^{-4},10^{-5}\}$ to demonstrate its effectiveness and efficiency. To access the accuracy of the AASG method and the MCM,  we obtain the reference solution $u_{\refs}(\pv,\sv)$ using the AASG method with a tolerance of $\tol=10^{-6}$.

\begin{table}[ht]
	\caption{Performance of the AASG method for test problem 1 with $N=50$.}\label{tab:diff2a}
	{\normalsize
		\begin{center}
			\begin{tabular}{lccccccccccccc}
				\toprule
				$\tol$ & $\vert\mathfrak{J}_1\vert$  &$\vert\tilde{\mathfrak{J}}_1\vert$ & $\vert\mathfrak{J}_2\vert$ & $\vert \tilde{\mathfrak{J}}_2\vert$ & $\vert\mathfrak{J}_3\vert$  & $\vert \tilde{\mathfrak{J}}_3\vert$ & $\vert\mathfrak{J}_4\vert$  & $\vert \tilde{\mathfrak{J}}_4\vert$  & $k$ &$\vert{\mathfrak{M}}_k^{5^\dagger}\vert$ & CPU time\\
				\midrule				
				$10^{-1}$  &   50  &    1  &    0  &    0  &    0  &    0  &    0  &    0  & 1  &   251  & 0.39\\
				$10^{-2}$  &   50  &   11  &   55  &    0  &    0  &    0  &    0  &    0  & 2  &   801  & 1.99\\
				$10^{-3}$  &   50  &   30  &  435  &    0  &    0  &    0  &    0  &    0  & 2  &  4601  & 16.37\\
				$10^{-4}$  &   50  &   50  &  1225  &   15  &    8  &    0  &    0  &    0  & 3  & 12581  & 84.39\\
				$10^{-5}$  &   50  &   50  &  1225  &   83  &  120  &    0  &    0  &    0  & 3  & 13701  & 88.55\\
				$10^{-6}$  &   50  &   50  &  1225  &  377  &  1537  &   15  &    1  &    0  & 4  & 27876  & 284.41\\
				\bottomrule
			\end{tabular}
		\end{center}
	}
\end{table}

Table \ref{tab:diff2a} presents the number of active indices for each order of ANOVA decomposition and the total number of selected gPC basis functions in the stochastic space. Furthermore, we report the computational time required for solving all linear systems that arise during the execution of the while loop in Algorithm~\ref{alg:adaptivestoch}. It can be observed that the number of selected gPC basis functions increases as the tolerance $\tol$ decreases, and therefore, the accuracy can be improved by reducing $\tol$. It is worth noting that the number of gPC basis functions with the total degree up to $p=5$ is $C_{55}^5=3478761$, which renders the standard stochastic Galerkin method practically infeasible for solving this problem in reasonable time.

\begin{figure}[!htbp]
	\begin{center}
		\subfloat[$\tol = 10^{-2}$.]{\includegraphics[height=0.27\linewidth]{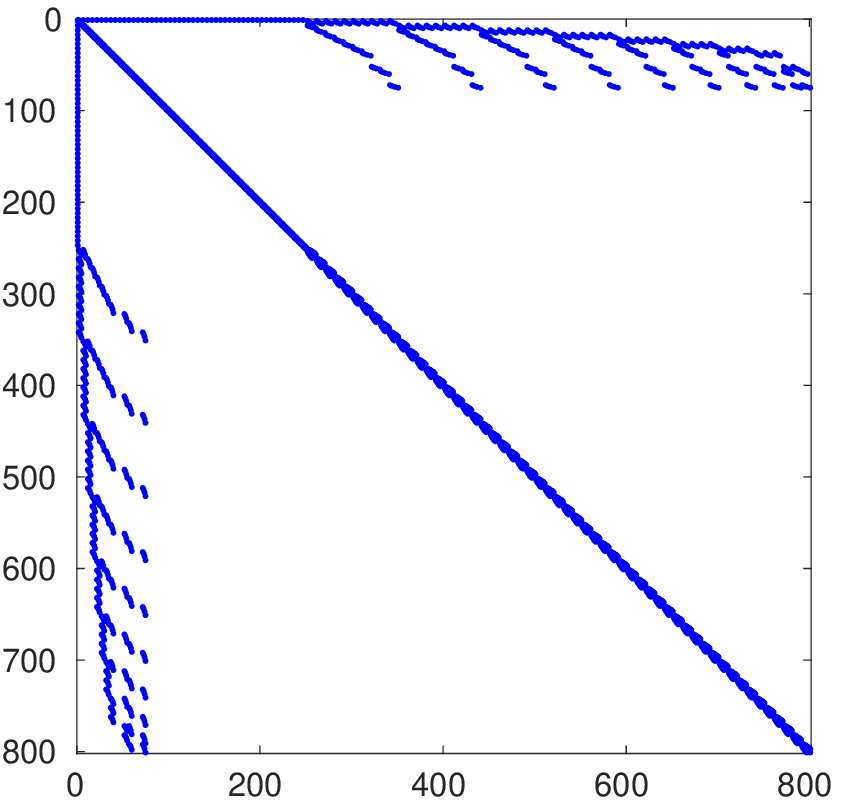}} \qquad\quad
		\subfloat[$\tol = 10^{-3}$.]{\includegraphics[height=0.27\linewidth]{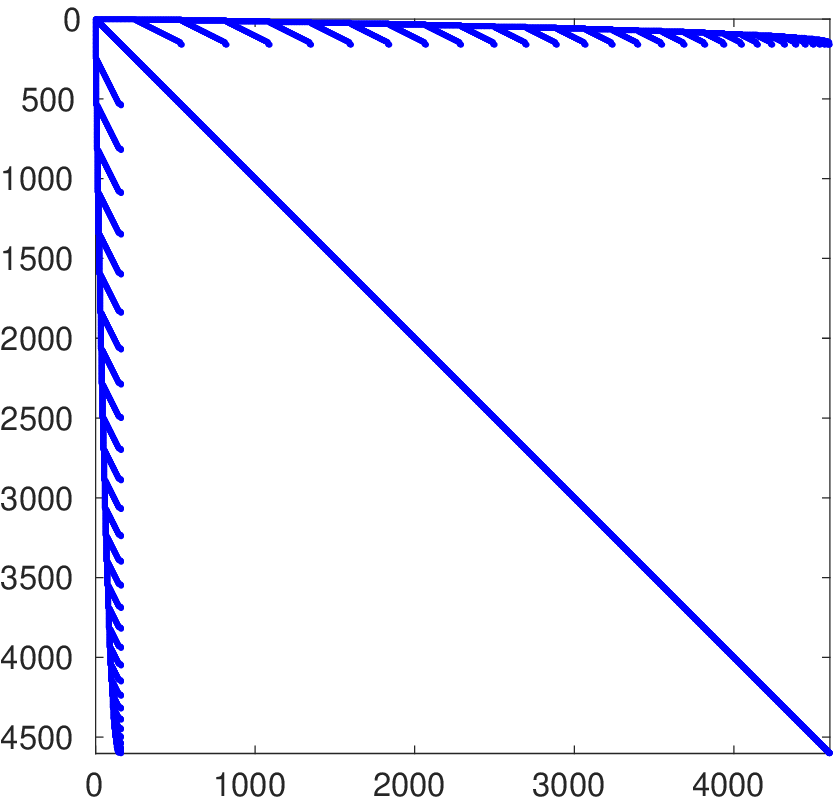}} \qquad\quad
		\subfloat[$\tol = 10^{-4}$.]{\includegraphics[height=0.27\linewidth]{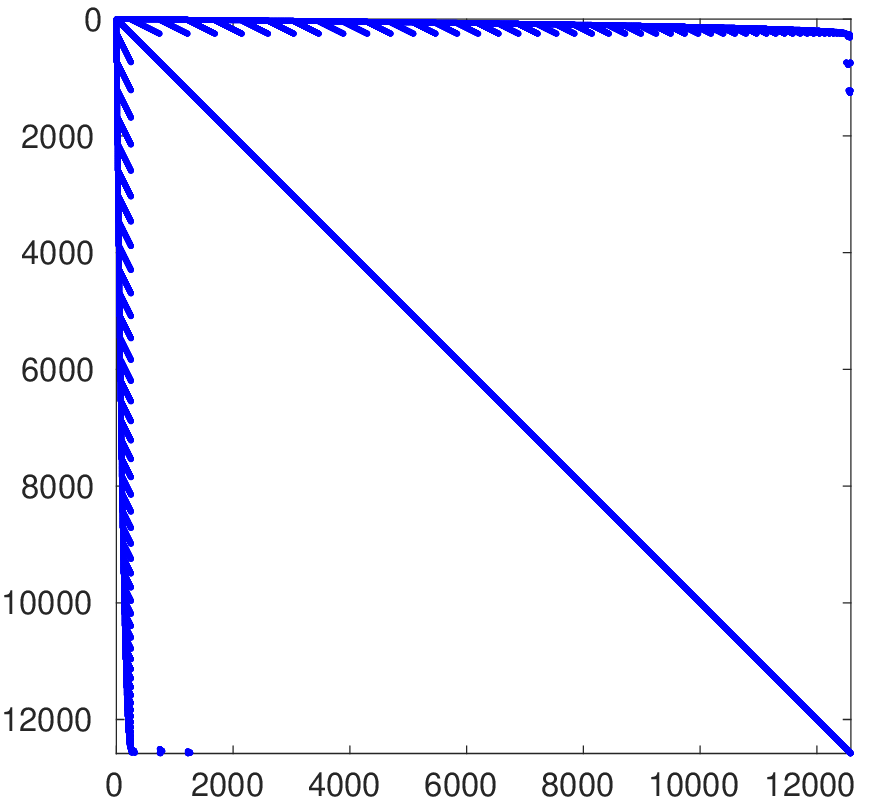}}
		\caption{Matrix block-structure (each block has dimension $\nphy\times\nphy$) for test problem 1 with $N=50$.}\label{fig:diff2a}
	\end{center}
\end{figure}

Fig. \ref{fig:diff2a} displays the block structure of the coefficient matrix of the resulting linear system. Each point in the figure represents a block of dimension $\nphy\times\nphy$. Furthermore, each nonzero block of the coefficient matrix has the same sparsity pattern as the corresponding deterministic problem. It is evident from the figure that the coefficient matrix in this case exhibits a sparser pattern than that of Case I.

\begin{figure}[!htbp]
	\begin{center}		
		\subfloat[Mean errors w.r.t. CPU time.]{\includegraphics[width=0.46\linewidth]{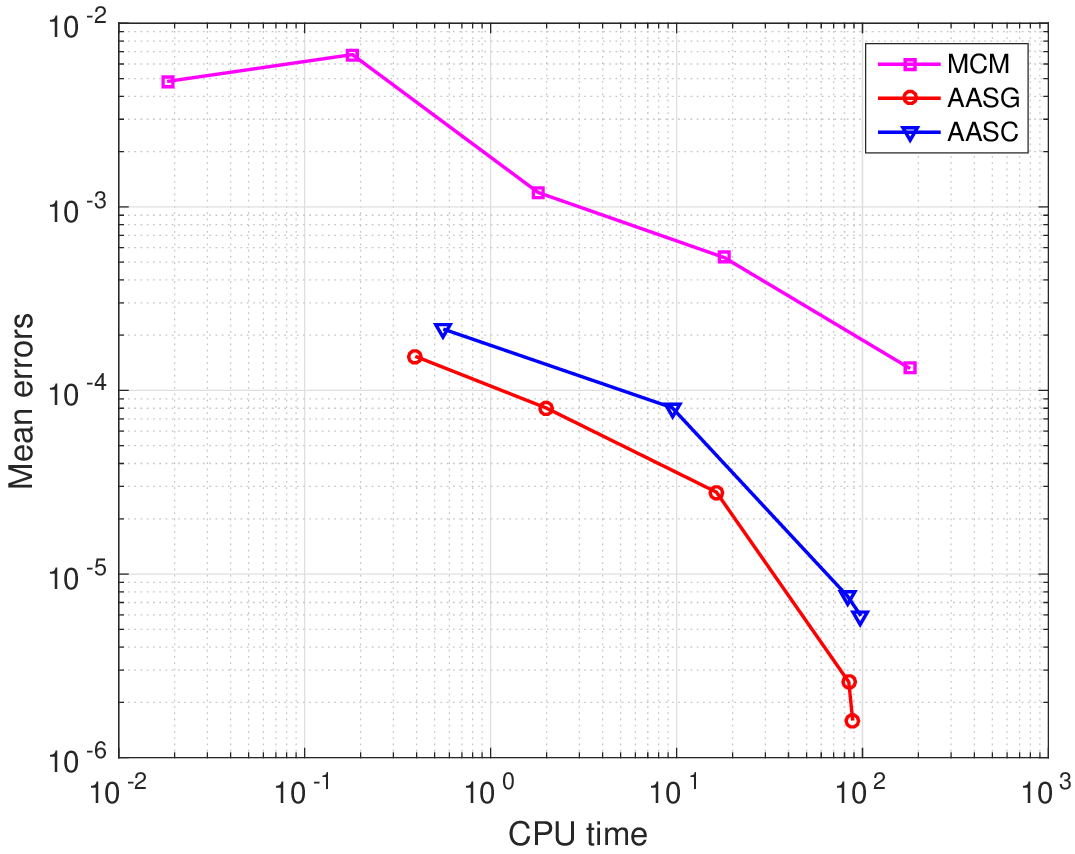}}\qquad
		\subfloat[Variance errors w.r.t. CPU time.]{\includegraphics[width=0.46\linewidth]{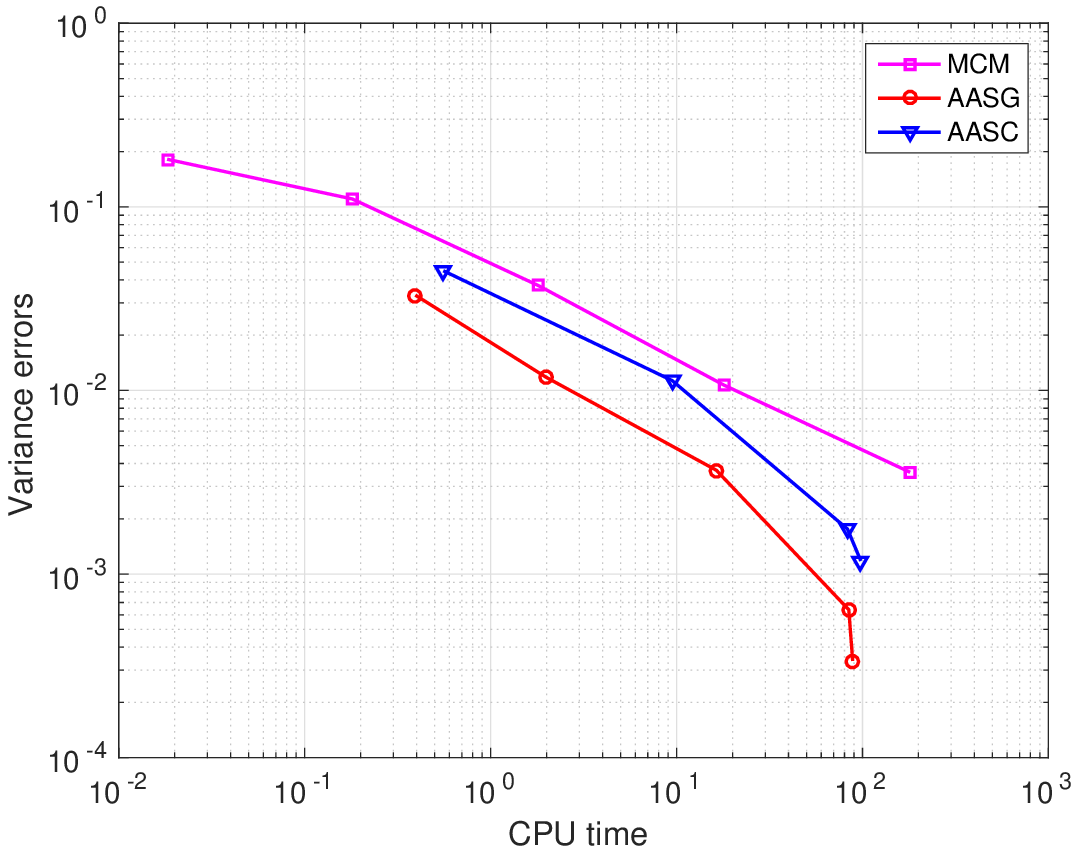}}\\
		\subfloat[Mean errors w.r.t. stochastic DOF.]{\includegraphics[width=0.46\linewidth]{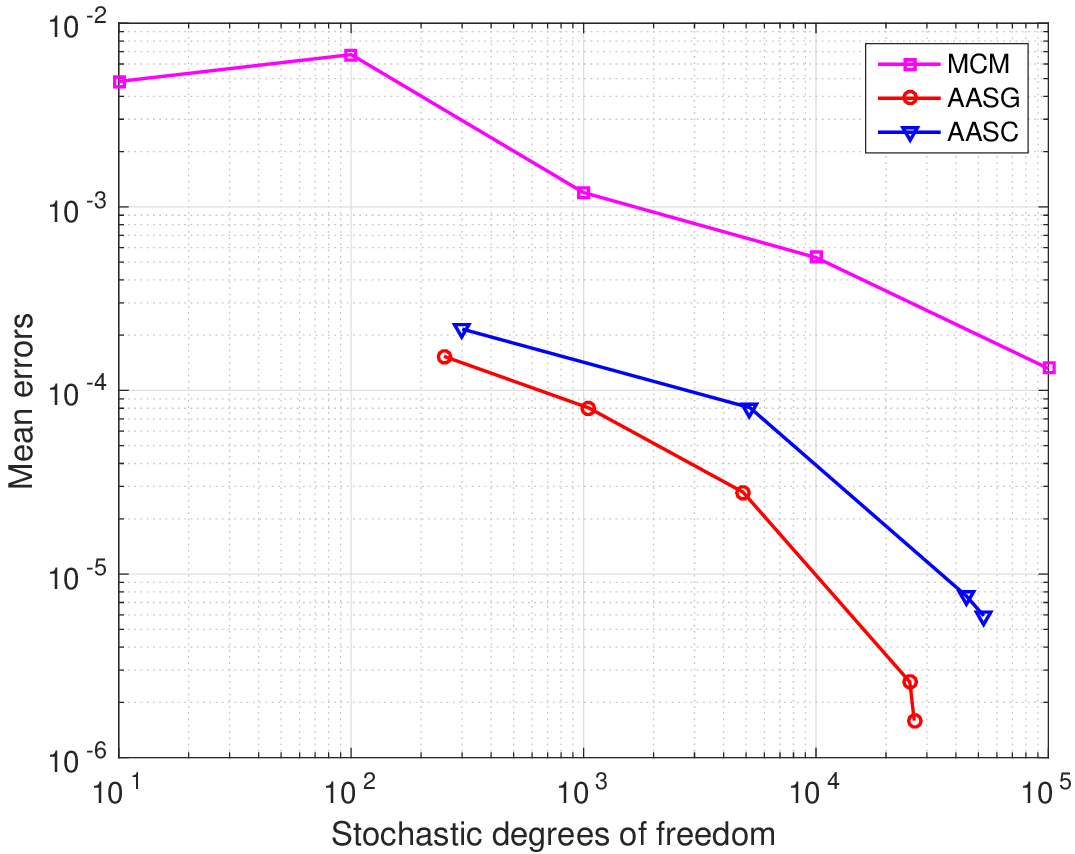}}\qquad
		\subfloat[Variance errors w.r.t. stochastic DOF.]{\includegraphics[width=0.46\linewidth]{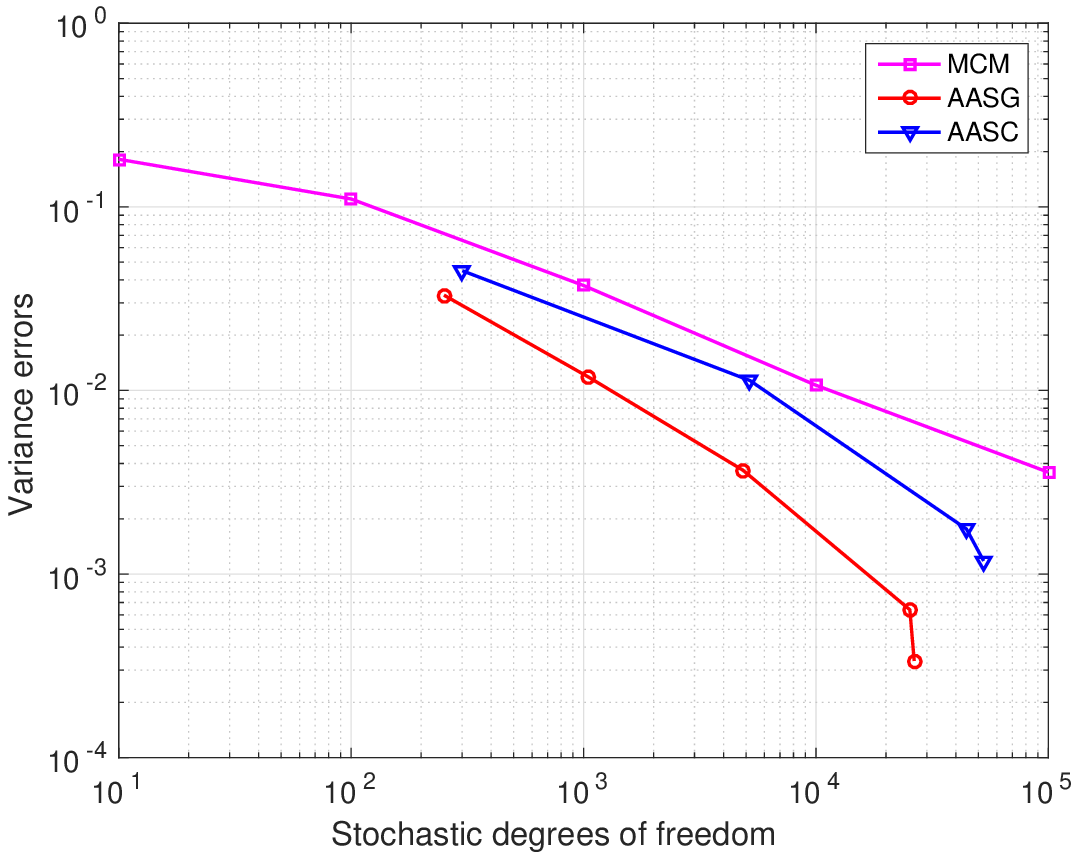}}
	\end{center}
	\caption{\guanjie{Comparison of errors with respect to CPU times and stochastic degrees of freedom for test problem 1 with $N=50$, where both the total degree of gPC in the AASG method and the grid level in the AASC method are set to $5$.}}\label{fig:diff2b}
\end{figure}

\guanjie{The plot displayed in Fig. \ref{fig:diff2b} compares the CPU times and the stochastic degrees of freedom required by the three methods in relation to mean and variance errors. The results clearly indicate that both the AASG method and the AASC method are significantly more efficient than the MCM, with improvements observed in both CPU time and stochastic degrees of freedom. Additionally, the AASG method outperforms the AASC method in terms of CPU time, and it is also evident that the AASG method provides a higher accuracy per stochastic degree of freedom compared to the AASC method.}

\subsection{Test problem 2}

In this test problem, we consider the stochastic Helmholtz problem given by
\begin{equation}\notag
\nabla^2u + a^2(\pv,\sv)u=f(\pv) \quad \mbox{in}\quad D\times\Gamma,
\end{equation}
with Sommerfeld radiation boundary condition.
Here, $D = [0,1]^2$ is the domain of interest and the Helmholtz coefficient $a(\pv,\sv)$ is a truncated KL expansion of a random field with a mean function $a_0(\pv)$, a standard deviation $\sigma=2\pi$, and the covariance function $\mathrm{Cov}\,(\bm{x},\bm{y})$ given by
\begin{equation}\notag
\mathrm{Cov}\,(\bm{x},\bm{y})=\sigma^2 \exp\left(-\frac{|x_1-y_1|}{c}-\frac{|x_2-y_2|}{c}
\right),
\end{equation}
where $\pv=[x_1,x_2]^T$, $\bm{y}=[y_1,y_2]^T$ and $c=1$ is the correlation length. Note that the KL expansion takes the form
\begin{equation}\label{eq:kl2}
a(\pv,\sv)=a_0(\pv)+\sum_{i=1}^{N}a_i(\pv)\svc_i=a_0(\pv)+\sum_{i=1}^{N}\sqrt{\lambda_i}c_i(\pv)\svc_i,
\end{equation}
where $a_0(\pv)=4\cdot(2\pi)$, $\{\lambda_i,c_i(\pv)\}_{i=1}^{N}$ are the eigenpairs of  $\mathrm{Cov}\,(\pv,\bm{y})$,
$\{\svc_i\}^{N}_{i=1}$ are uncorrelated random variables, and
$N$ is the number of KL modes retained. The  Gaussian point source at the center of the domain is used as the source term, i.e.,
\begin{equation}\notag
	f(\pv) = \mbox{e}^{-(8\cdot4)^2((x_1-0.5)^2+(x_2-0.5)^2)}
\end{equation}

For this test problem, we assume that the random variables  $\{\svc_i\}^{N}_{i=1}$ are independent and uniformly distributed within the range $[-1,1]$. The parameters of $a(\pv,\sv)$ are set as shown in Table~\ref{tab:helm_set}. We use the perfectly matched layers (PML) to simulate the Sommerfeld condition~\cite{Berenger94}, and generate the matrices $\{\bm{A}_i\}_{i=1}^K$ using the codes associated with  \cite{Liu2015Additive}. In the physical domain, the meshgrid is set to $33\times33$ (i.e., $\nphy=33$). In the stochastic space, the total degree of the gPC basis functions in the AASG method is set to $p=6$. The linear systems arising from both the standard stochastic Galerkin method and the AASG method are solved using the preconditioned bi-conjugate gradient stabilized (Bi-CGSTAB) method, with a tolerance of $10^{-8}$ and the mean based preconditioner~\cite{Powell2009}. Furthermore, for the MCM, the linear systems is solved using the MATLAB backslash solver.

\begin{table}[ht]%
	\caption{Parameters of the Helmholtz coefficient $a(\pv,\sv)$ in \eqref{eq:kl2} and $K$ in \eqref{eq:affine1}.}\label{tab:helm_set}
	\begin{center}%
		\newcolumntype{C}{>{\centering\arraybackslash}X}%
		\begin{tabularx}{0.9\linewidth}{CCCCCC}%
			\toprule
			Case   		&  $N$ & $K$ \\%
			\midrule
			I           & $4$ & $25$\\%
			II   		& $10$ & $121$\\%
			\bottomrule%
		\end{tabularx}%
	\end{center}%
\end{table}%

\subsubsection{Case I: a 4 dimensional Helmholtz problem}
We consider the AASG method with decreasing tolerances $\tol = \{10^{-1},10^{-2},10^{-3},10^{-4},10^{-5}\}$ to demonstrate its effectiveness and efficiency. To access the accuracy of the AASG method and the MCM,  we obtain the reference solution $u_{\refs}(\pv,\sv)$ using  
the standard stochastic Galerkin method with the total degree of up to $p=8$.

\begin{table}[!htbp]
	\caption{Performance of the AASG method for test problem 2 with $N=4$.}\label{tab:helm1a}
	{\normalsize
		\begin{center}
			\begin{tabular}{lccccccccccc}
				\toprule
				$\tol$ & $\vert\mathfrak{J}_1\vert$  &$\vert\tilde{\mathfrak{J}}_1\vert$ 
				& $\vert\mathfrak{J}_2\vert$  & $\vert \tilde{\mathfrak{J}}_2\vert$ 
				& $\vert\mathfrak{J}_3\vert$  & $\vert \tilde{\mathfrak{J}}_3\vert$ 
				& $\vert\mathfrak{J}_4\vert$  & $\vert \tilde{\mathfrak{J}}_4\vert$ 
				& $k$ &$\vert{\mathfrak{M}}_k^{6^\dagger}\vert$ & CPU time\\
				\midrule
				$10^{-1}$  &    4  &    1  &    0  &    0  &    0  &    0  &    0  &    0  & 1  &    25  & 0.22\\
				$10^{-2}$  &    4  &    4  &    6  &    3  &    0  &    0  &    0  &    0  & 2  &   115  & 1.52\\
				$10^{-3}$  &    4  &    4  &    6  &    5  &    2  &    1  &    0  &    0  & 3  &   155  & 3.62\\
				$10^{-4}$  &    4  &    4  &    6  &    6  &    4  &    3  &    0  &    0  & 3  &   195  & 4.46\\
				$10^{-5}$  &    4  &    4  &    6  &    6  &    4  &    4  &    1  &    1  & 4  &   210  & 7.95\\
				\bottomrule
				
			\end{tabular}
		\end{center}
	}
\end{table}

Table \ref{tab:helm1a} presents the number of active indices for each order of ANOVA decomposition and the total number of selected gPC basis functions in the stochastic space. Furthermore, we report the computational time required for solving all linear systems that arise during the execution of the while loop in Algorithm~\ref{alg:adaptivestoch}. It can be observed that the number of selected gPC basis functions increases as the tolerance $\tol$ decreases. For a $4$ dimensional problem, there are $C_{10}^6=210$ gPC basis functions with the total degree up to $6$.  From the table, it can be seen that when $\tol=10^{-5}$, all the gPC basis functions are selected in the AASG method,  making the AASG method equivalent to the standard stochastic Galerkin method.

\begin{figure}[!htbp]
	\begin{center}
		\subfloat[$\tol = 10^{-2}$.]{\includegraphics[height=0.27\linewidth]{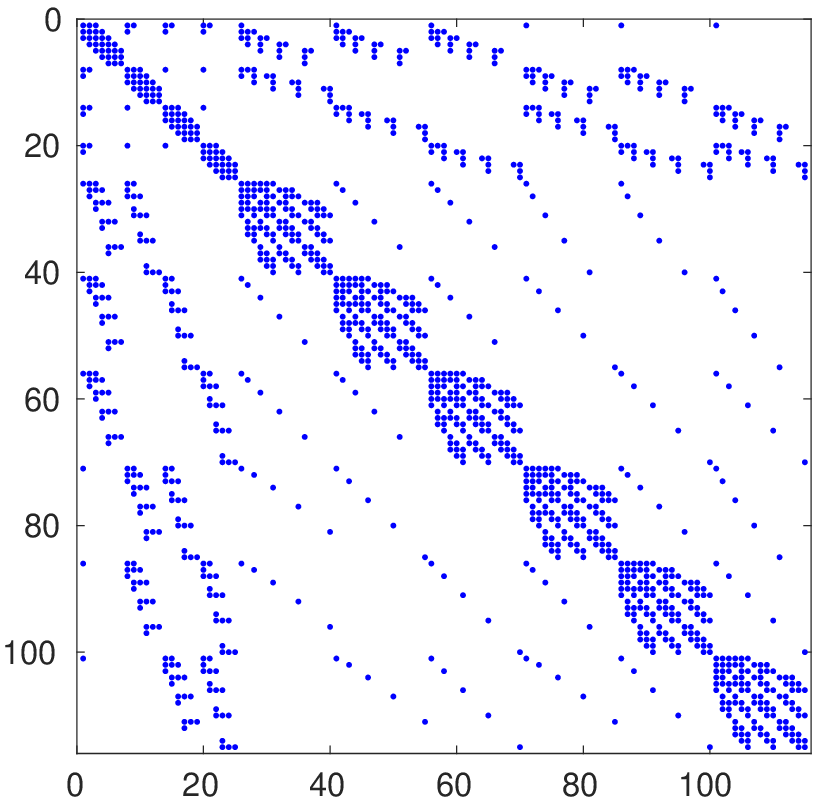}} \qquad\quad
		\subfloat[$\tol = 10^{-3}$.]{\includegraphics[height=0.27\linewidth]{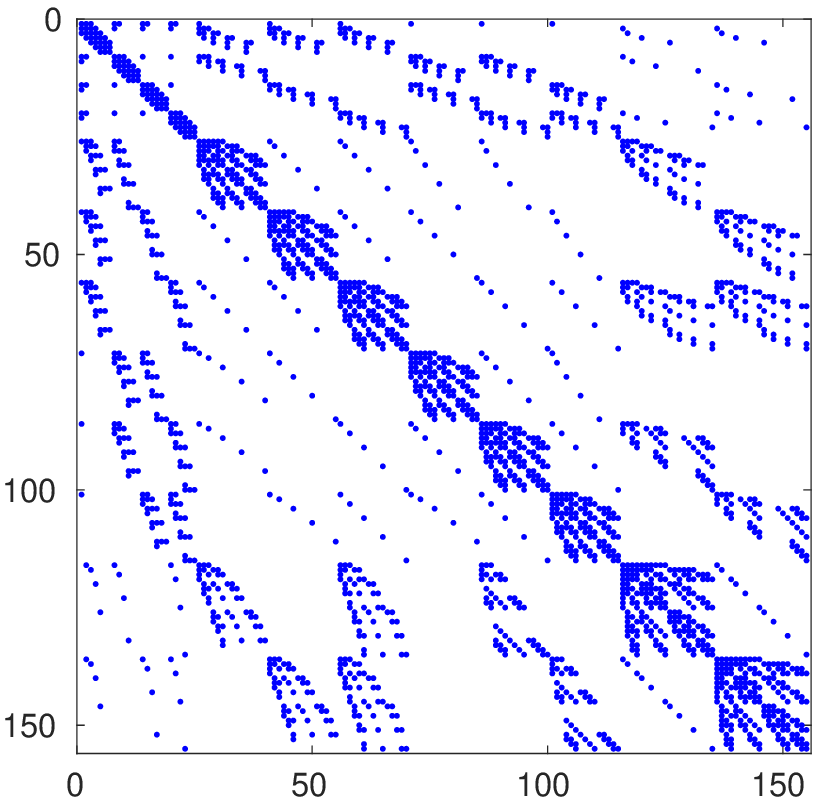}} \qquad\quad
		\subfloat[$\tol = 10^{-4}$.]{\includegraphics[height=0.27\linewidth]{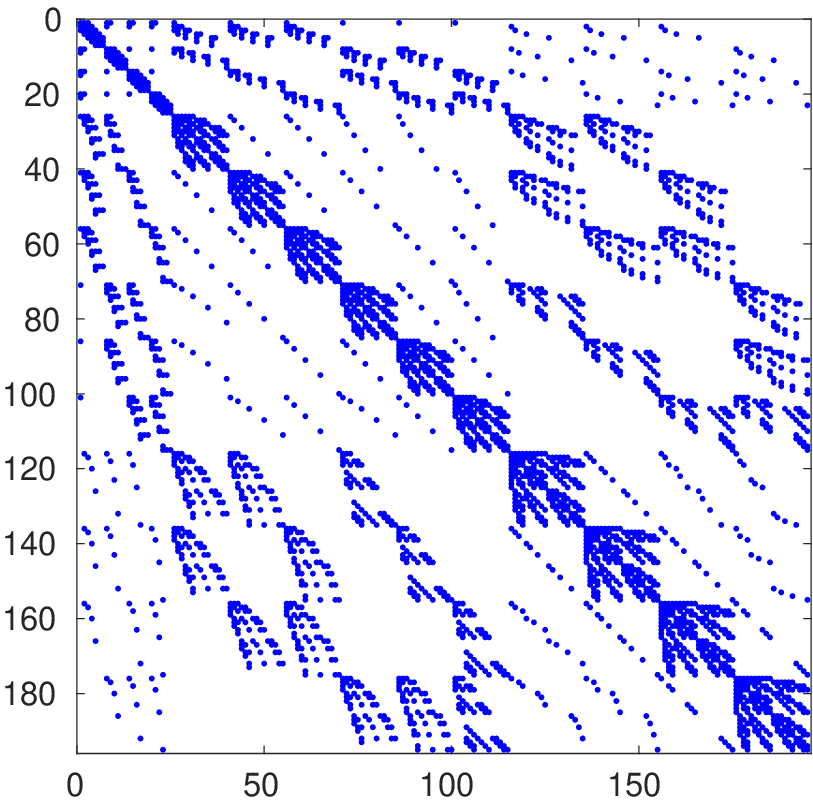}}
		\caption{Matrix block-structure (each block has dimension $\nphy\times\nphy$) for test problem 2 with $N=4$.}\label{fig:helm1b}
	\end{center}
\end{figure}

Fig.~\ref{fig:helm1b} illustrates the block structure of the coefficient matrix of the resulting linear system. Each point in the figure represents a block of dimension $\nphy\times\nphy$. Moreover, each nonzero block of the coefficient matrix has the same sparsity pattern as the corresponding deterministic problem. Although the coefficient matrix of the Helmholtz equation is much denser than that of the diffusion equation, it is still very sparse and thus should be solved by iterative methods.

\begin{figure}[!htbp]
	\begin{center}		
		\subfloat[Mean errors w.r.t. CPU time.]{\includegraphics[width=0.46\linewidth]{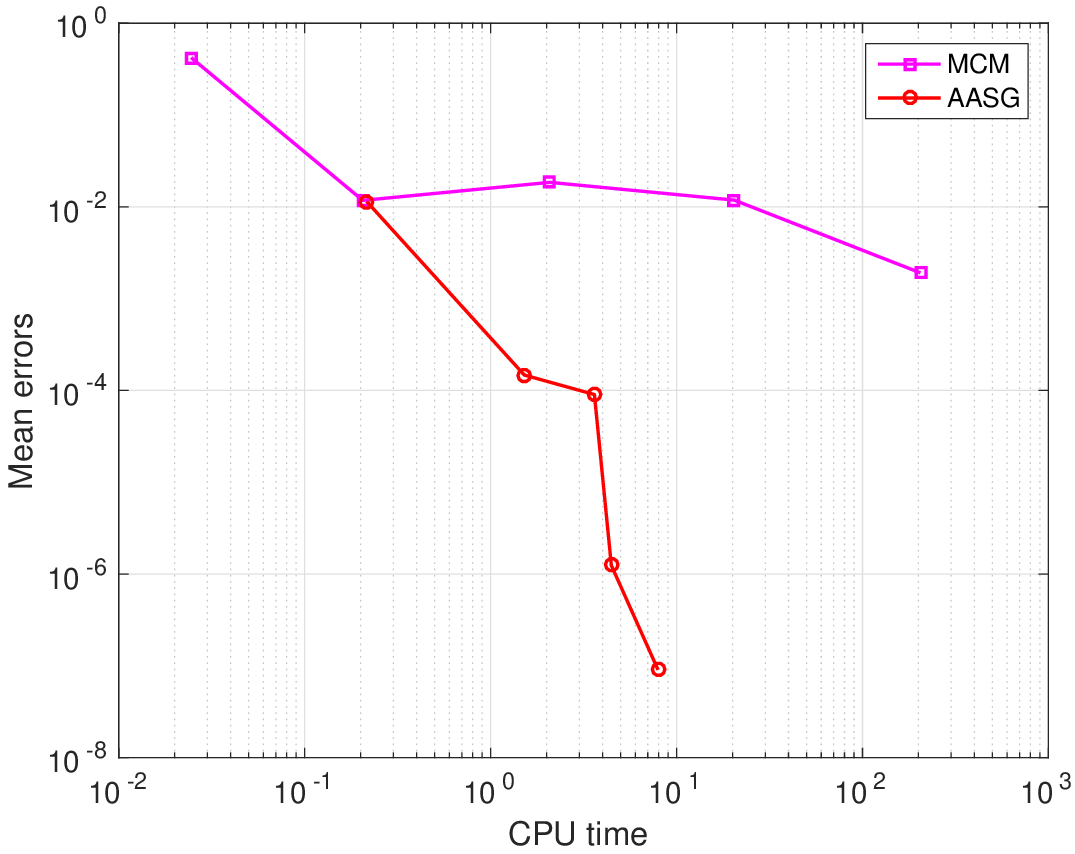}}\qquad
		\subfloat[Variance errors w.r.t. CPU time.]{\includegraphics[width=0.46\linewidth]{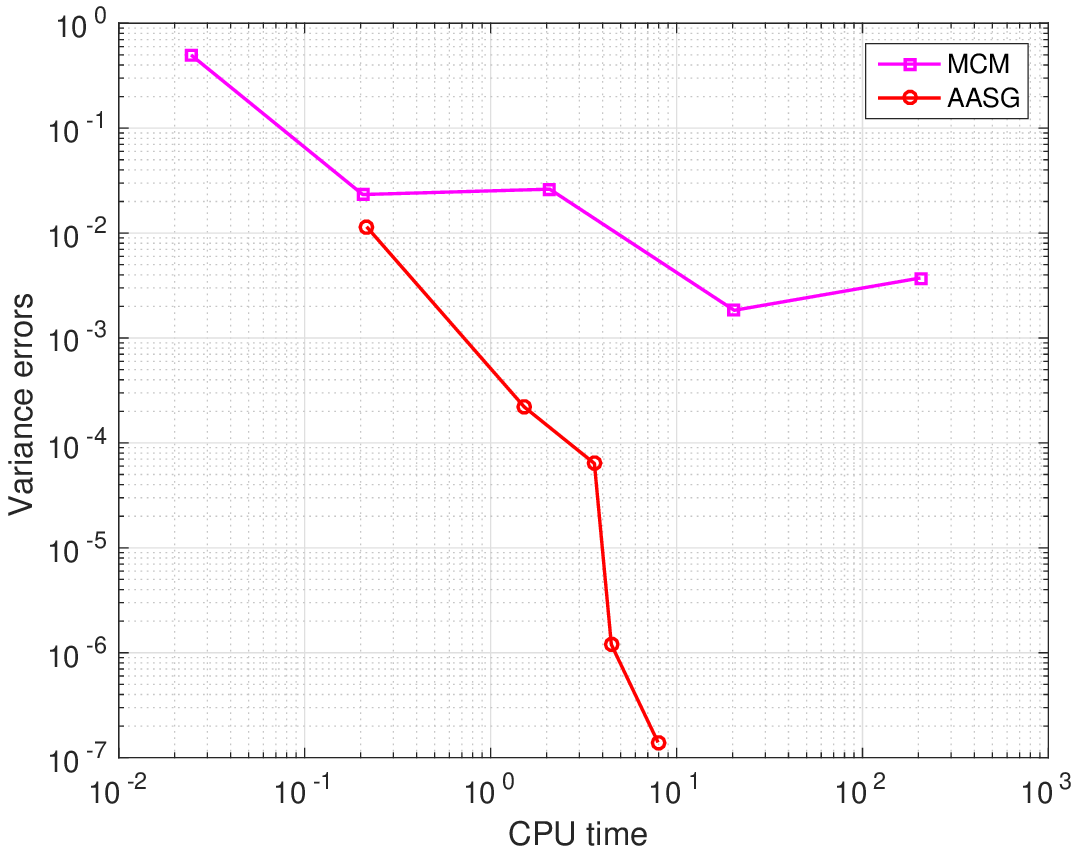}}\\
		\subfloat[Mean errors w.r.t. DOF.]{\includegraphics[width=0.46\linewidth]{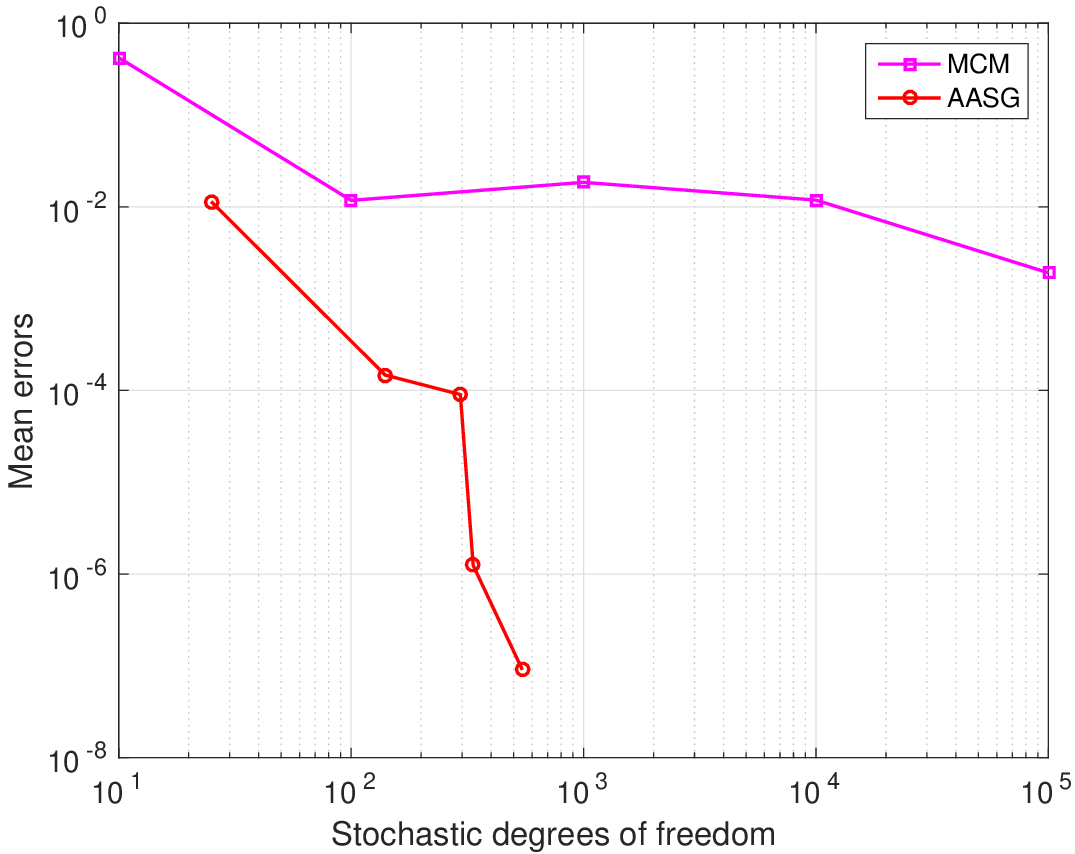}}\qquad
		\subfloat[Variance errors w.r.t. DOF.]{\includegraphics[width=0.46\linewidth]{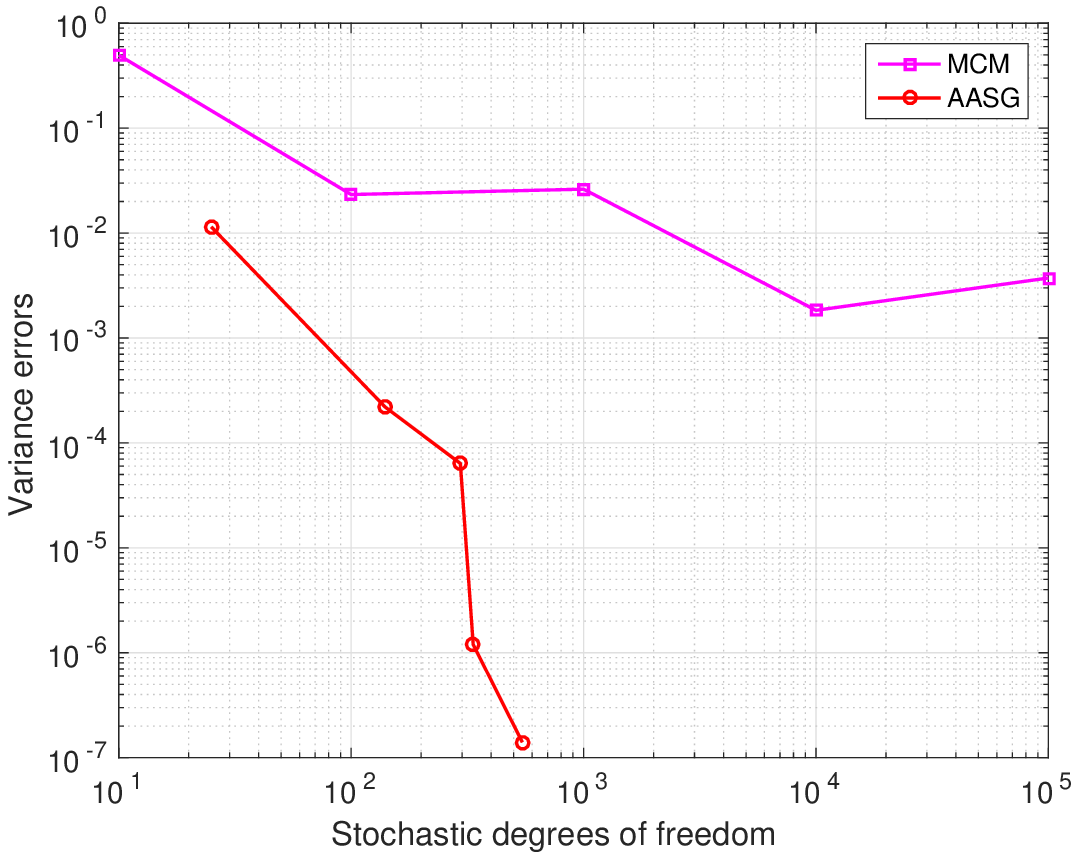}}
	\end{center}
	\caption{\guanjie{Comparison of errors with respect to CPU times and stochastic degrees of freedom for test problem 2 with $N=4$, where the total degree of gPC in the AASG method is set to $6$.}}\label{fig:helm1c}
\end{figure}

\guanjie{Fig. \ref{fig:helm1c} investigates the accuracy achieved by the two methods, presenting errors concerning both CPU time and stochastic degrees of freedom. For clarity, CPU time denotes the total time required to solve all the linear systems within the respective procedures. For the AASG method, stochastic degrees of freedom encompass the cumulative count of gPC basis functions generated during the execution of the while loop in Algorithm~\ref{alg:adaptivestoch}, while for the MCM, it corresponds to the total number of sample points used. The results indicate the notable efficiency of the AASG method in comparison to the MCM, as it is dozens of times faster in terms of CPU time and hundreds of times faster in terms of stochastic degrees of freedom. This distinction from the diffusion problem is noteworthy, as the performance of CPU time does not seem to align with that of stochastic degrees of freedom. This phenomenon can be attributed to the fact that the number of matrix-vector products computed in each iteration for solving the linear system~\eqref{eq:linsg} is proportional to $(N+1)^2$ for Helmholtz problems, while it is proportional to $N+1$ for diffusion problems.}

\subsubsection{Case II: a 10 dimensional Helmholtz problem}
We consider the AASG method  with decreasing tolerances $\tol = \{10^{-1},10^{-2},10^{-3},10^{-4},10^{-5}\}$ to demonstrate its effectiveness and efficiency. To access the accuracy of the AASG method and the MCM,  we obtain the reference solution $u_{\refs}(\pv,\sv)$ using  
the standard stochastic Galerkin method with the total degree of up to $p=8$.

\begin{table}[!htbp]
	\caption{Performance of the AASG method for test problem 2 with $N=10$.}\label{tab:helm2a}
	{\normalsize
		\begin{center}
			\begin{tabular}{lcccccccccccc}
				\toprule
				$\tol$ & $\vert\mathfrak{J}_1\vert$  &$\vert\tilde{\mathfrak{J}}_1\vert$ 
				& $\vert\mathfrak{J}_2\vert$  & $\vert \tilde{\mathfrak{J}}_2\vert$ 
				& $\vert\mathfrak{J}_3\vert$  & $\vert \tilde{\mathfrak{J}}_3\vert$ 
				& $\vert\mathfrak{J}_4\vert$  & $\vert \tilde{\mathfrak{J}}_4\vert$ 
				& $k$ &$\vert{\mathfrak{M}}_k^{6^\dagger}\vert$ & CPU time\\
				\midrule				
				$10^{-1}$  &   10  &    1  &    0  &    0  &    0  &    0  &    0  &    0  &    1  &    61  & 0.72\\
				$10^{-2}$  &   10  &    9  &   36  &    6  &    0  &    0  &    0  &    0  &    2  &   601  & 14.03\\
				$10^{-3}$  &   10  &   10  &   45  &   14  &    7  &    3  &    0  &    0  &    3  &   876  & 41.77\\
				$10^{-4}$  &   10  &   10  &   45  &   32  &   55  &   21  &    0  &    0  &    3  &  1836  & 82.92\\
				$10^{-5}$  &   10  &   10  &   45  &   43  &  105  &   59  &   41  &   22  &    4  &  3451  & 281.31\\
				\bottomrule
			\end{tabular}
		\end{center}
	}
\end{table}

Table \ref{tab:helm2a} presents the number of active indices for each order of ANOVA decomposition and the total number of selected gPC basis functions in the stochastic space. Furthermore, we report the computational time required for solving all linear systems that arise during the execution of the while loop in Algorithm~\ref{alg:adaptivestoch}. It can be observed that the number of selected gPC basis functions increases as the tolerance $\tol$ decreases, and therefore, the accuracy can be improved by reducing $\tol$.

\begin{figure}[!htbp]
	\begin{center}
		\subfloat[$\tol = 10^{-2}$.]{\includegraphics[height=0.27\linewidth]{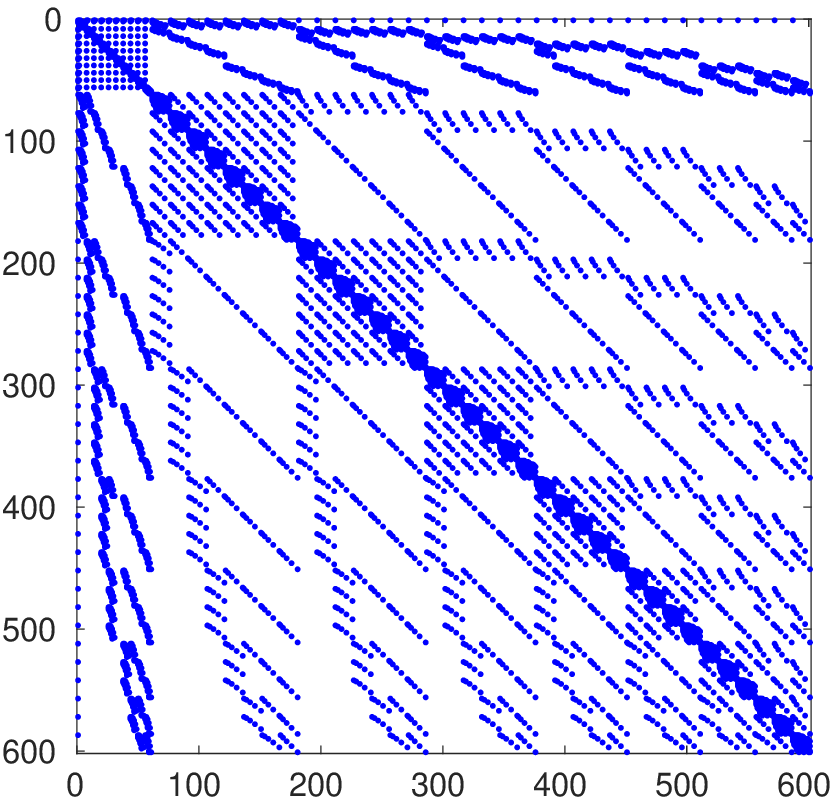}} \qquad\quad
		\subfloat[$\tol = 10^{-3}$.]{\includegraphics[height=0.27\linewidth]{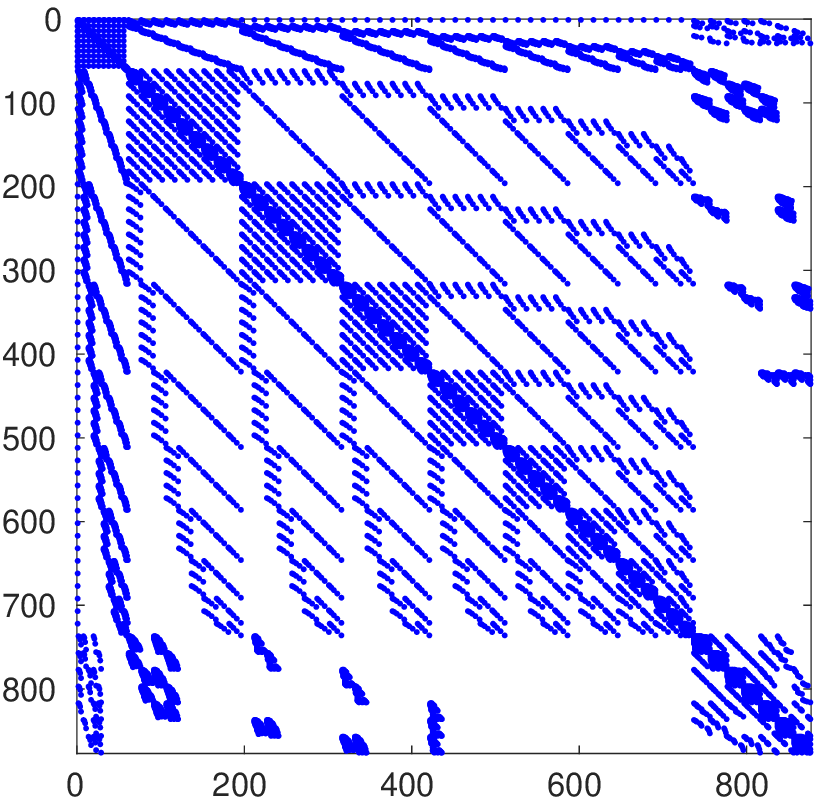}} \qquad\quad
		\subfloat[$\tol = 10^{-4}$.]{\includegraphics[height=0.27\linewidth]{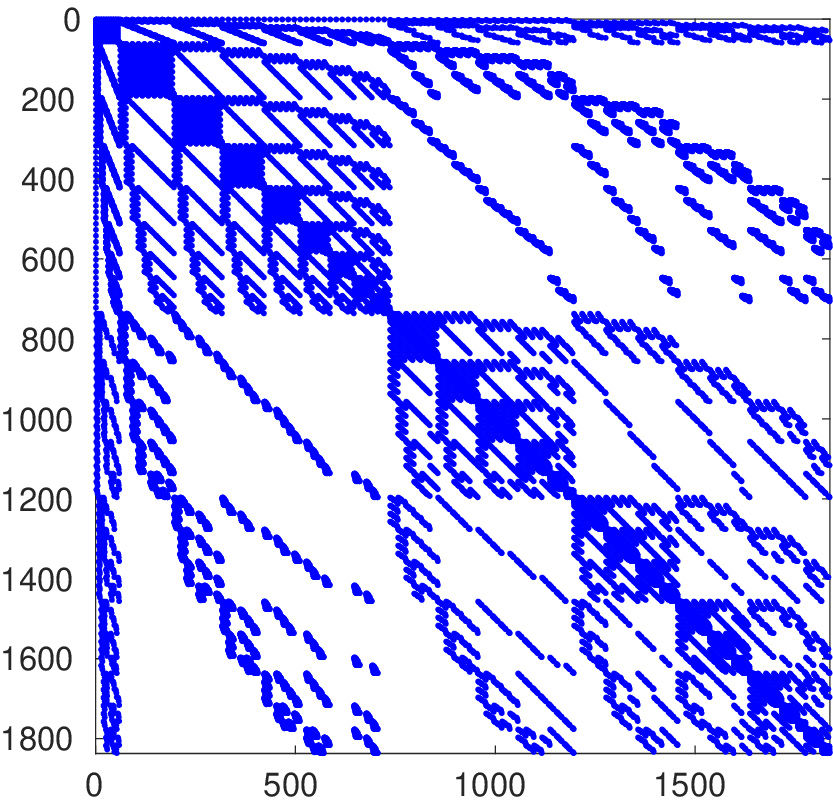}}
		\caption{Matrix block-structure (each block has dimension $\nphy\times\nphy$) for test problem 2 with $N=10$.}\label{fig:helm2b}
	\end{center}
\end{figure}

Fig.~\ref{fig:helm2b} illustrates the block structure of the coefficient matrix of the resulting linear system. Each point in the figure represents a block of dimension $\nphy\times\nphy$. Moreover, each nonzero block of the coefficient matrix has the same sparsity pattern as the corresponding deterministic problem. It can be observed that the coefficient matrix of the Helmholtz equation is much denser than that of the diffusion equation, which makes it more time consuming to solve than the diffusion problem.

\begin{figure}[!htbp]
	\begin{center}		
		\subfloat[Mean errors w.r.t. CPU time.]{\includegraphics[width=0.46\linewidth]{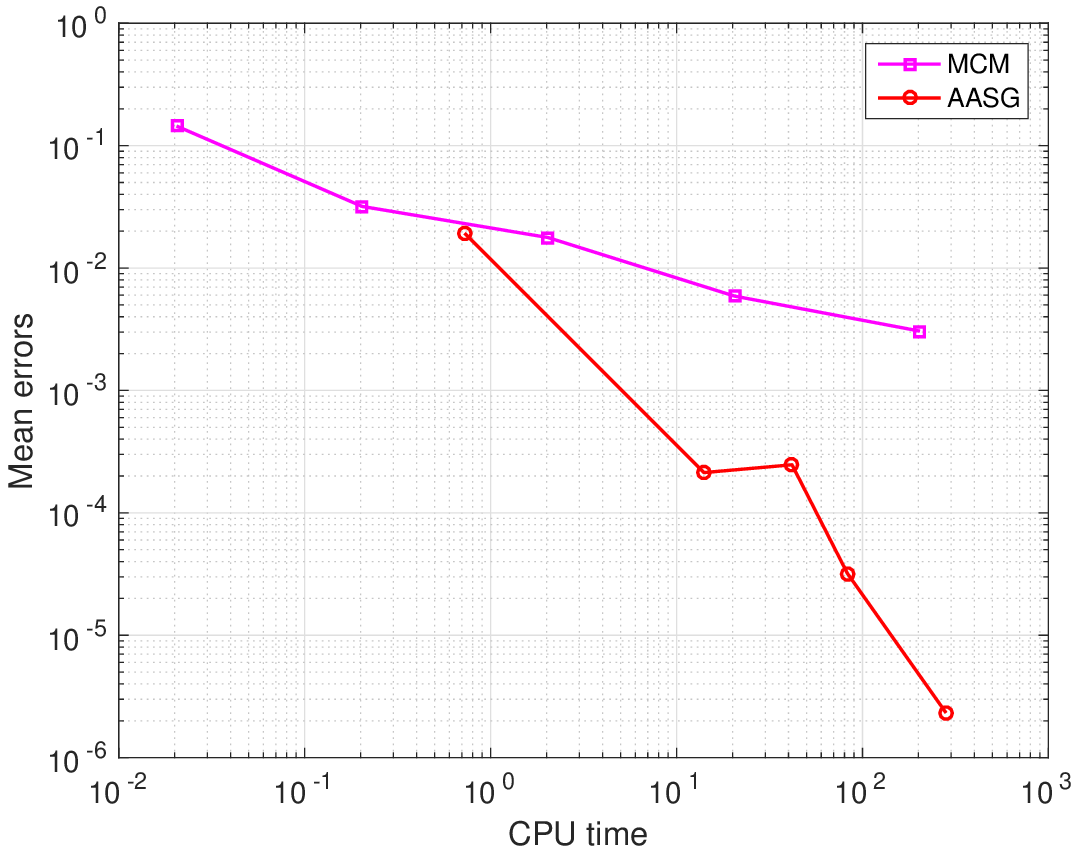}}\qquad
		\subfloat[Variance errors w.r.t. CPU time.]{\includegraphics[width=0.46\linewidth]{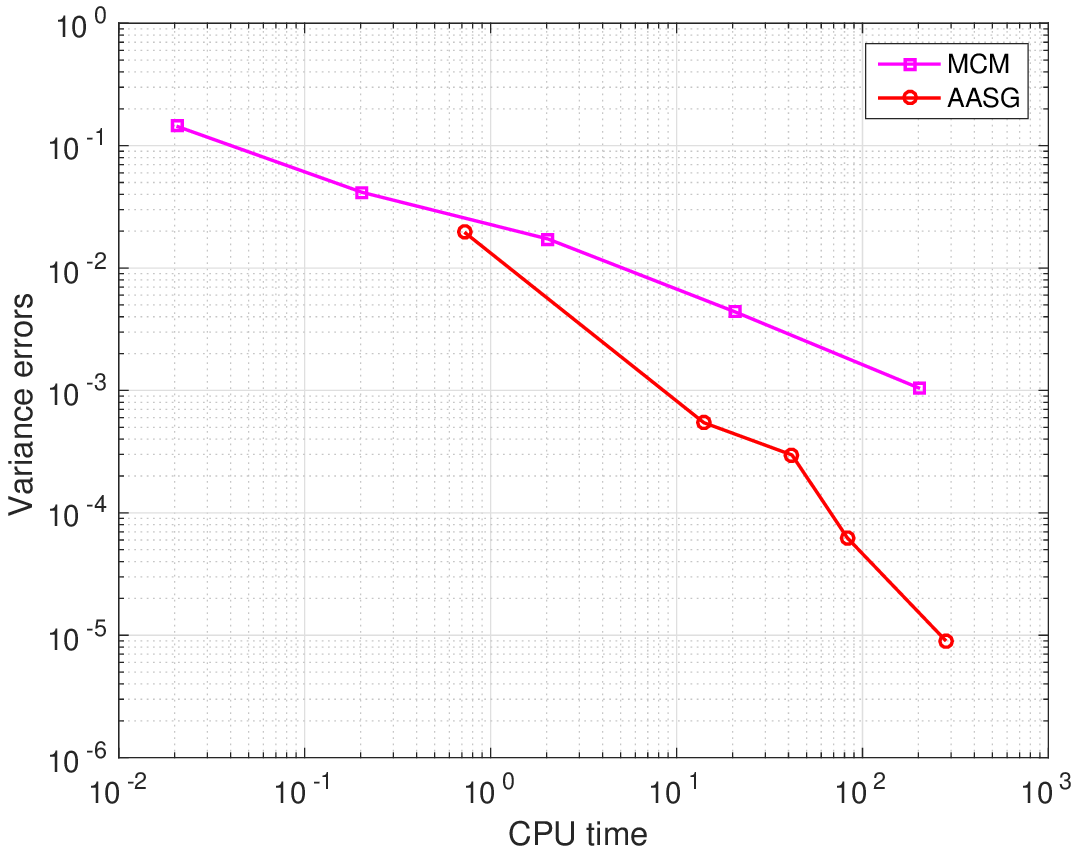}}\\
		\subfloat[Mean errors w.r.t. DOF.]{\includegraphics[width=0.46\linewidth]{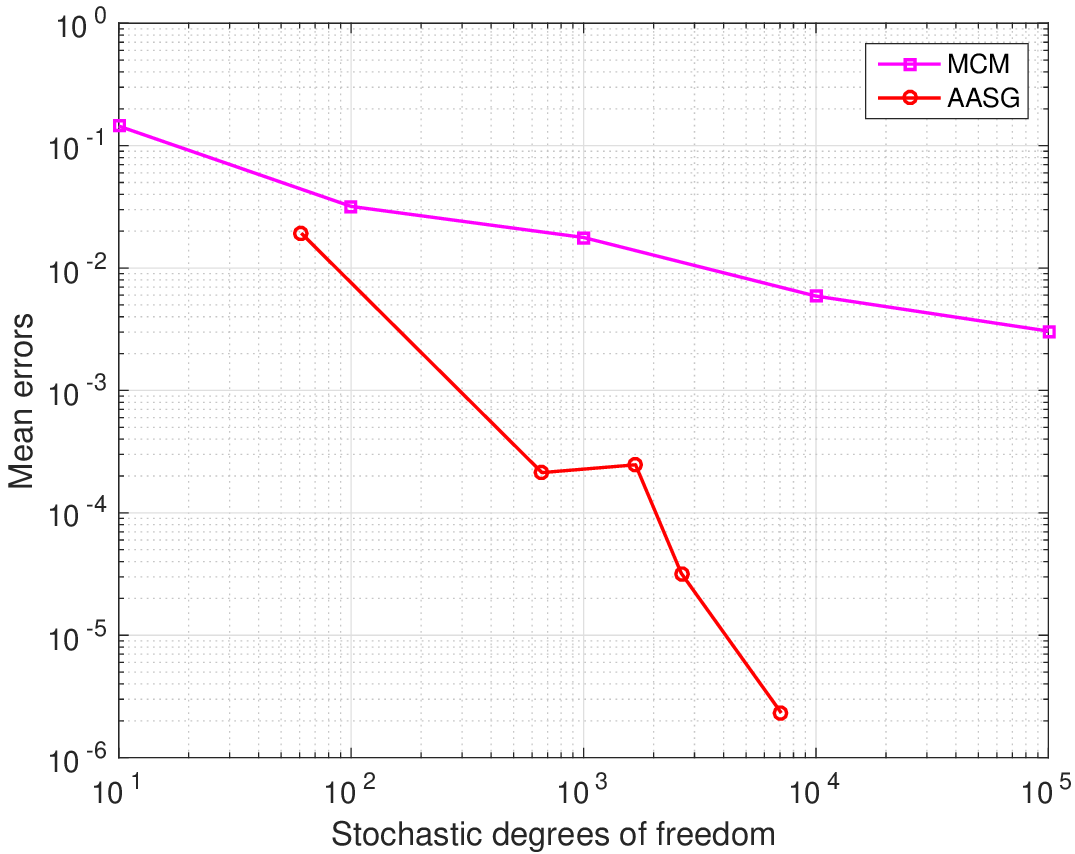}}\qquad
		\subfloat[Variance errors w.r.t. DOF.]{\includegraphics[width=0.46\linewidth]{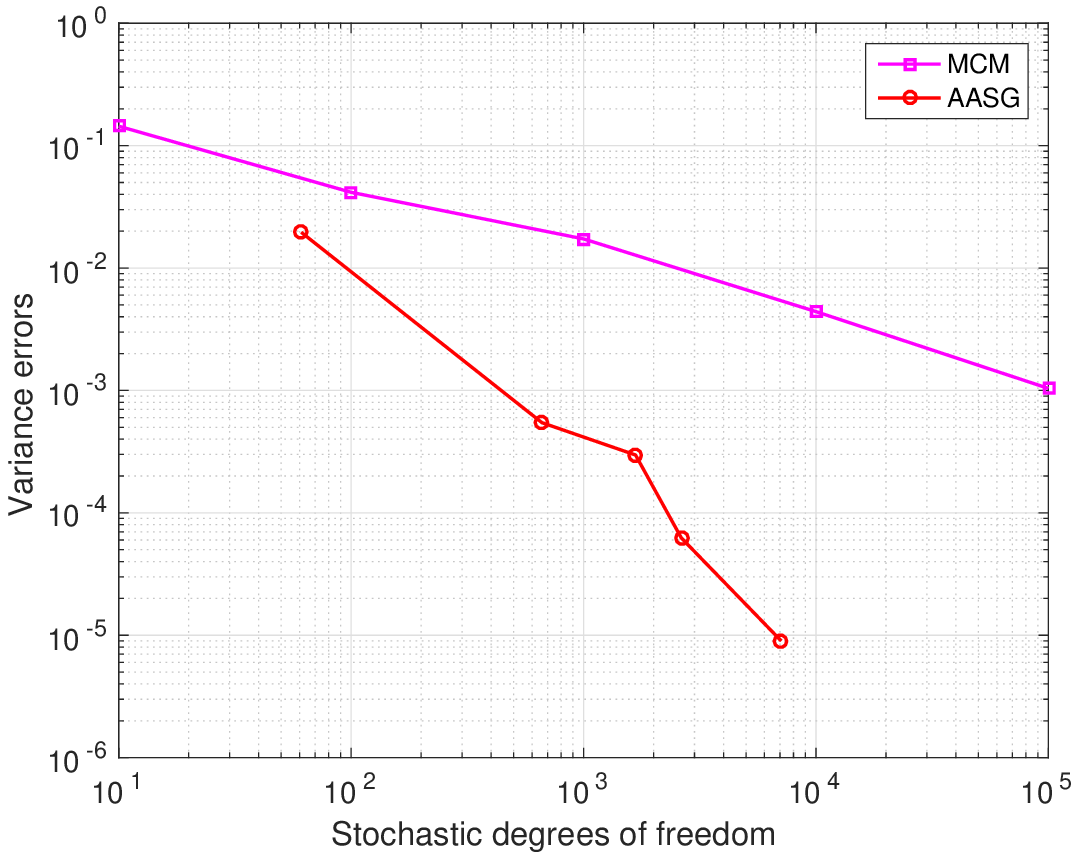}}
	\end{center}
	\caption{\guanjie{Comparison of errors with respect to CPU times and stochastic degrees of freedom for test problem 2 with $N=10$, where the total degree of gPC in the AASG method is set to $6$.}}\label{fig:helm2c}
\end{figure}

\guanjie{Fig. \ref{fig:helm2c} investigates the accuracy achieved by the two methods, presenting errors concerning both CPU time and stochastic degrees of freedom. The results indicate the notable efficiency of the AASG method in comparison to the MCM, as it is dozens of times faster in terms of CPU time and hundreds of times faster in terms of stochastic degrees of freedom. This distinction from the diffusion problem is noteworthy, as the performance of CPU time does not seem to align with that of stochastic degrees of freedom.  This phenomenon can be attributed to the fact that the number of matrix-vector products computed in each iteration for solving the linear system~\eqref{eq:linsg} is proportional to $(N+1)^2$ for Helmholtz problems, in contrast to $N+1$ for diffusion problems.}

\section{Conclusion}\label{sec:concusion}
In this work, we investigate the generalized polynomial chaos (gPC) expansion of component functions for the ANOVA decomposition, and present a concise form of the gPC expansion for each component function. With this formulation, we propose an adaptive ANOVA stochastic Galerkin method for solving partial differential equations with random inputs. The proposed method effectively selects basis functions in the stochastic space, enabling significant reduction in the dimension of the stochastic approximation space. \guanjie{Compared with anchored ANOVA methods, the proposed approach avoids the difficulty for selecting proper anchor points, which are crucial for achieving efficient approximations in the context of anchored ANOVA methods.} Numerical simulations are conducted to demonstrate the effectiveness and efficiency of the proposed method. While our current focus is on selecting the basis in the stochastic space, future work will explore techniques for reducing computational costs in the physical space.
\newline

\bibliographystyle{spmpsci}       
\bibliography{AASG_V2}                

\begin{thebibliography}{10}
\providecommand{\url}[1]{{#1}}
\providecommand{\urlprefix}{URL }
\expandafter\ifx\csname urlstyle\endcsname\relax
  \providecommand{\doi}[1]{DOI~\discretionary{}{}{}#1}\else
  \providecommand{\doi}{DOI~\discretionary{}{}{}\begingroup
  \urlstyle{rm}\Url}\fi

\bibitem{Agarwal2009domain}
Agarwal, N., Aluru, N.R.: A domain adaptive stochastic collocation approach for
  analysis of {MEMS} under uncertainties.
\newblock Journal of Computational Physics \textbf{228}(20), 7662--7688 (2009).
\newblock \doi{10.1016/j.jcp.2009.07.014}

\bibitem{Askey1975orthogonal}
Askey, R.: Orthogonal polynomials and special functions.
\newblock SIAM (1975)

\bibitem{Babuska2004}
Babu\v{s}ka, I., Tempone, R.l., Zouraris, G.E.: Galerkin finite element
  approximations of stochastic elliptic partial differential equations.
\newblock SIAM Journal on Numerical Analysis \textbf{42}(2), 800--825 (2004).
\newblock \doi{10.1137/S0036142902418680}

\bibitem{Berenger94}
Berenger, J.P.: A perfectly matched layer for the absorption of electromagnetic
  waves.
\newblock Journal of Computational Physics \textbf{114}(2), 185--200 (1994).
\newblock \doi{10.1006/jcph.1994.1159}

\bibitem{Caflisch1998}
Caflisch, R.E.: {Monte Carlo and quasi-Monte Carlo} methods.
\newblock Acta Numerica \textbf{7}, 1--49 (1998)

\bibitem{Cheng2013a}
Cheng, M., Hou, T.Y., Zhang, Z.: A dynamically bi-orthogonal method for
  time-dependent stochastic partial differential equations {I}: {Derivation}
  and algorithms.
\newblock Journal of Computational Physics \textbf{242}, 843--868 (2013).
\newblock \doi{10.1016/j.jcp.2013.02.033}

\bibitem{Cheng2013b}
Cheng, M., Hou, T.Y., Zhang, Z.: A dynamically bi-orthogonal method for
  time-dependent stochastic partial differential equations {II}: {Adaptivity}
  and generalizations.
\newblock Journal of Computational Physics \textbf{242}, 753--776 (2013).
\newblock \doi{0.1016/j.jcp.2013.02.020}

\bibitem{cho2018adaptive}
Cho, H., Elman, H.C.: An adaptive reduced basis collocation method based on
  {PCM ANOVA} decomposition for anisotropic stochastic {PDEs}.
\newblock International Journal for Uncertainty Quantification \textbf{8}(3)
  (2018).
\newblock \doi{10.1615/Int.J.UncertaintyQuantification.2018024436}

\bibitem{Elman2007}
Elman, H., Furnival, D.: Solving the stochastic steady-state diffusion problem
  using multigrid.
\newblock IMA Journal of Numerical Analysis \textbf{27}(4), 675--688 (2007).
\newblock \doi{10.1093/imanum/drm006}

\bibitem{Elmanliao}
Elman, H., Liao, Q.: Reduced basis collocation methods for partial differential
  equations with random coefficients.
\newblock SIAM/ASA Journal on Uncertainty Quantification \textbf{1}, 192--217
  (2013).
\newblock \doi{10.1137/120881841}

\bibitem{Elman2005}
Elman, H.C., Ernst, O.G., O’Leary, D.P., Stewart, M.: Efficient iterative
  algorithms for the stochastic finite element method with application to
  acoustic scattering.
\newblock Computer Methods in Applied Mechanics and Engineering \textbf{194},
  1037--1055 (2005).
\newblock \doi{10.1016/j.cma.2004.06.028}

\bibitem{Feng2015}
Feng, X., Lin, J., Lorton, C.: An efficient numerical method for acoustic wave
  scattering in random media.
\newblock SIAM/ASA Journal on Uncertainty Quantification \textbf{3}(1),
  790--822 (2015).
\newblock \doi{10.1137/140958232}

\bibitem{fishman2013}
Fishman, G.: Monte Carlo: concepts, algorithms, and applications.
\newblock Springer Science \& Business Media (2013)

\bibitem{Gao2010anova}
Gao, Z., Hesthaven, J.S.: On {ANOVA} expansions and strategies for choosing the
  anchor point.
\newblock Applied Mathematics and Computation \textbf{217}(7), 3274--3285
  (2010).
\newblock \doi{10.1016/j.amc.2010.08.061}

\bibitem{Ghanem2003}
Ghanem, R.G., Spanos, P.D.: Stochastic finite elements: a spectral approach.
\newblock Courier Corporation (2003)

\bibitem{guo2020constructing}
Guo, L., Narayan, A., Zhou, T.: Constructing least-squares polynomial
  approximations.
\newblock SIAM Review \textbf{62}(2), 483--508 (2020).
\newblock \doi{10.1137/18M1234151}

\bibitem{jakeman2017generalized}
Jakeman, J.D., Narayan, A., Zhou, T.: A generalized sampling and
  preconditioning scheme for sparse approximation of polynomial chaos
  expansions.
\newblock SIAM Journal on Scientific Computing \textbf{39}(3), A1114--A1144
  (2017).
\newblock \doi{10.1137/16M1063885}

\bibitem{Fabian2022}
K{\"a}mmerer, L., Potts, D., Taubert, F.: The uniform sparse {FFT} with
  application to {PDEs} with random coefficients.
\newblock Sampling Theory, Signal Processing, and Data Analysis \textbf{20}(19)
  (2021).
\newblock \doi{10.1007/s43670-022-00037-3}

\bibitem{LeeElman16}
Lee, K., Elman, H.C.: A preconditioned low-rank projection method with a
  rank-reduction scheme for stochastic partial differential equations.
\newblock SIAM Journal on Scientific Computing \textbf{39}(5), S828--S850
  (2017).
\newblock \doi{10.1137/16M1075582}

\bibitem{Lee2019low}
Lee, K., Elman, H.C., Sousedik, B.: A low-rank solver for the {Navier--Stokes}
  equations with uncertain viscosity.
\newblock SIAM/ASA Journal on Uncertainty Quantification \textbf{7}(4),
  1275--1300 (2019).
\newblock \doi{10.1137/17M1151912}

\bibitem{Liao2016reduced}
Liao, Q., Lin, G.: Reduced basis {ANOVA} methods for partial differential
  equations with high-dimensional random inputs.
\newblock Journal of Computational Physics \textbf{317}, 148--164 (2016).
\newblock \doi{10.1016/j.jcp.2016.04.029}

\bibitem{Liu2015Additive}
Liu, F., Ying, L.: Additive sweeping preconditioner for the {Helmholtz}
  equation.
\newblock Multiscale Modeling \& Simulation \textbf{14}(2), 799--822 (2016).
\newblock \doi{10.1137/15M1017144}

\bibitem{Ma2010adaptive}
Ma, X., Zabaras, N.: An adaptive high-dimensional stochastic model
  representation technique for the solution of stochastic partial differential
  equations.
\newblock Journal of Computational Physics \textbf{229}(10), 3884--3915 (2010).
\newblock \doi{10.1016/j.jcp.2010.01.033}

\bibitem{musharbash2015error}
Musharbash, E., Nobile, F., Zhou, T.: Error analysis of the dynamically
  orthogonal approximation of time dependent random {PDEs}.
\newblock SIAM Journal on Scientific Computing \textbf{37}(2), A776--A810
  (2015).
\newblock \doi{10.1137/140967787}

\bibitem{Daniel2021}
Potts, D., Schmischke, M.: Approximation of high-dimensional periodic functions
  with {Fourier}-based methods.
\newblock SIAM Journal on Numerical Analysis \textbf{59}(5), 2393--2429 (2021).
\newblock \doi{10.1137/20M1354921}

\bibitem{Powell2009}
Powell, C.E., Elman, H.C.: Block-diagonal preconditioning for spectral
  stochastic finite element systems.
\newblock IMA Journal of Numerical Analysis \textbf{29}(2), 350--375 (2009).
\newblock \doi{10.1093/imanum/drn014}

\bibitem{powell2015}
Powell, C.E., Silvester, D., Simoncini, V.: An efficient reduced basis solver
  for stochastic {Galerkin} matrix equations.
\newblock SIAM Journal on Scientific Computing \textbf{39}(1), A141--A163
  (2017).
\newblock \doi{10.1137/15M1032399}

\bibitem{Sobol2003theorems}
Sobol', I.M.: Theorems and examples on high dimensional model representation.
\newblock Reliability Engineering and System Safety \textbf{79}(2), 187--193
  (2003).
\newblock \doi{10.1016/S0951-8320(02)00229-6}

\bibitem{Sudret2008global}
Sudret, B.: Global sensitivity analysis using polynomial chaos expansions.
\newblock Reliability Engineering and System Safety \textbf{93}(7), 964--979
  (2008).
\newblock \doi{10.1016/j.ress.2007.04.002}

\bibitem{Tang2015sensitivity}
Tang, K., Congedo, P.M., Abgrall, R.: Sensitivity analysis using anchored
  {ANOVA} expansion and high-order moments computation.
\newblock International Journal for Numerical Methods in Engineering
  \textbf{102}(9), 1554--1584 (2015).
\newblock \doi{10.1002/nme.4856}

\bibitem{Tang2016}
Tang, K., Congedo, P.M., Abgrall, R.: Adaptive surrogate modeling by {ANOVA}
  and sparse polynomial dimensional decomposition for global sensitivity
  analysis in fluid simulation.
\newblock Journal of Computational Physics \textbf{314}(1), 557--589 (2016).
\newblock \doi{10.1016/j.jcp.2016.03.026}

\bibitem{tang2010convergence}
Tang, T., Zhou, T.: Convergence analysis for stochastic collocation methods to
  scalar hyperbolic equations with a random wave speed.
\newblock Commun. Comput. Phys \textbf{8}(1), 226--248 (2010).
\newblock \doi{10.4208/cicp.060109.130110a}

\bibitem{Wan2005adaptive}
Wan, X., Karniadakis, G.E.: An adaptive multi-element generalized polynomial
  chaos method for stochastic differential equations.
\newblock Journal of Computational Physics \textbf{209}(2), 617--642 (2005).
\newblock \doi{10.1016/j.jcp.2005.03.023}

\bibitem{Wang2008approximation}
Wang, X.: On the approximation error in high dimensional model representation.
\newblock In: 2008 Winter Simulation Conference, pp. 453--462. IEEE (2008).
\newblock \doi{10.1109/WSC.2008.4736100}

\bibitem{williamson2021application}
Williamson, K., Cho, H., Soused{\'\i}k, B.: Application of adaptive {ANOVA} and
  reduced basis methods to the stochastic {Stokes-Brinkman} problem.
\newblock Computational Geosciences \textbf{25}(3), 1191--1213 (2021).
\newblock \doi{10.1007/s10596-021-10048-z}

\bibitem{Xiu2010}
Xiu, D.: Numerical methods for stochastic computations: a spectral method
  approach.
\newblock Princeton University Press (2010)

\bibitem{Xiu05}
Xiu, D., Hesthaven, J.: High-order collocation methods for differential
  equations with random inputs.
\newblock SIAM Journal on Scientific Computing \textbf{27}, 1118--1139 (2005).
\newblock \doi{10.1137/040615201}

\bibitem{Xiu2002modeling}
Xiu, D., Karniadakis, G.E.: Modeling uncertainty in steady state diffusion
  problems via generalized polynomial chaos.
\newblock Computer Methods in Applied Mechanics and Engineering
  \textbf{191}(43), 4927--4948 (2002).
\newblock \doi{10.1016/S0045-7825(02)00421-8}

\bibitem{Xiu2002wiener}
Xiu, D., Karniadakis, G.E.: The {Wiener-Askey} polynomial chaos for stochastic
  differential equations.
\newblock SIAM Journal on Scientific Computing \textbf{24}(2), 619--644 (2002).
\newblock \doi{10.1137/S1064827501387826}

\bibitem{Xiu2003modeling}
Xiu, D., Karniadakis, G.E.: Modeling uncertainty in flow simulations via
  generalized polynomial chaos.
\newblock Journal of Computational Physics \textbf{187}(1), 137--167 (2003).
\newblock \doi{doi:10.1016/S0021-9991(03)00092-5}

\bibitem{yan2019adaptive}
Yan, L., Zhou, T.: Adaptive multi-fidelity polynomial chaos approach to
  {Bayesian} inference in inverse problems.
\newblock Journal of Computational Physics \textbf{381}, 110--128 (2019).
\newblock \doi{10.1016/j.jcp.2018.12.025}

\bibitem{Yang2012adaptive}
Yang, X., Choi, M., Lin, G., Karniadakis, G.E.: Adaptive {ANOVA} decomposition
  of stochastic incompressible and compressible flows.
\newblock Journal of Computational Physics \textbf{231}(4), 1587--1614 (2012)

\end{thebibliography}

\end{document}